\documentclass[11pt]{amsart}

\usepackage[utf8]{inputenc}
\usepackage{enumerate}
\usepackage{amsrefs}
\usepackage{dsfont}
\usepackage{amsmath}
\usepackage{amsthm}
\usepackage{amssymb}
\usepackage{amsfonts}
\usepackage{mathrsfs}  
\usepackage[shortlabels]{enumitem} 
\usepackage{color}
\usepackage{xcolor}
\usepackage[bookmarks=true,hyperindex,pdftex,colorlinks,citecolor=blue]{hyperref}
\usepackage{marginnote} 
\usepackage{graphicx} 
\usepackage{caption} 
\usepackage{soul} 
\usepackage{bm}
\usepackage[centertags]{amsmath}
\usepackage{tocvsec2}

\newcommand{\ee}{\varepsilon}
\newcommand{\nn}{\mathbb{N}}

\newcommand{\co}{\text{cof}}

\newcommand{\uuu}{\mathcal{U}}

\newcommand{\eps}{\varepsilon}

\newcommand{\Ndb}{\mathbb N}
\newcommand{\Qdb}{\mathbb Q}

\newcommand{\cal}{\mathcal}

\newcommand{\n}{\overline n}
\newcommand{\m}{\overline m}

\theoremstyle{plain}
\newtheorem{theorem}{Theorem}[section]
\newtheorem{lemma}[theorem]{Lemma}
\newtheorem{corollary}[theorem]{Corollary}
\newtheorem{proposition}[theorem]{Proposition}
\newtheorem{claim}{Claim}

\newtheorem{thmAlfa}{Theorem}

\theoremstyle{definition}
\newtheorem*{definition*}{Definition}
\newtheorem{definition}[theorem]{Definition}

\newtheorem{problem}[theorem]{Problem}

\theoremstyle{remark}
\newtheorem{remark}[theorem]{Remark}

\begin{document}
	
	
	\title[Asymptotic smoothness]{Asymptotic smoothness in Banach spaces, three-space properties and applications}
	
	\author{R.M.~Causey}
	\address{R.M~Causey}
	\email{rmcausey1701@gmail.com}

	\author{A.~Fovelle}
	\address{A.~Fovelle, Laboratoire de Math\'ematiques de Besan\c con, Universit\'e Bourgogne Franche-Comt\'e, 16 route de Gray, 25030 Besan\c con C\'edex, Besan\c con, France}
	\email{audrey.fovelle@univ-fcomte.fr}
	
	\author{G.~Lancien}
	\address{G.~Lancien, Laboratoire de Math\'ematiques de Besan\c con, Universit\'e Bourgogne Franche-Comt\'e, 16 route de Gray, 25030 Besan\c con C\'edex, Besan\c con, France}
	\email{gilles.lancien@univ-fcomte.fr}
	
	\thanks{The second-named and third-named authors received support from the EIPHI Graduate School (contract ANR-17-EURE-0002).}
	
	\begin{abstract} We study four asymptotic smoothness properties of Banach spaces, denoted $\textsf{T}_p,\textsf{A}_p, \textsf{N}_p$, and $\textsf{P}_p$. We complete their description by proving the missing renorming characterization for $\textsf{A}_p$. We show that asymptotic uniform flattenability (property $\textsf{T}_\infty$) and summable Szlenk index (property $\textsf{A}_\infty$) are three-space properties. Combined with the positive results of the first-named author, Draga, and Kochanek, and with the counterexamples we provide, this completely solves the three-space problem for this  family of properties. We also derive from our characterizations of $\textsf{A}_p$ and $\textsf{N}_p$ in terms of equivalent renormings, new coarse Lipschitz rigidity results for these classes. 
	\end{abstract}
	
	\maketitle
	
	\setcounter{tocdepth}{1}
	
	\section{Introduction}
	
	A coarse Lipschitz embedding from a metric space $M$ into a metric space $N$ is a map which is bi-Lipschitz for large enough distances (see precise definitions in Section 6). In the seminal paper \cite{Ribe}, M. Ribe proved that if a Banach space $X$ coarse Lipschitz embeds into a Banach space $Y$, then $X$ is finitely crudely representable into $Y$, which means that all finite-dimensional subspaces of $X$ are linearly isomorphic, with a uniform distortion, to subspaces of $Y$. In other words, the local properties of Banach spaces, that are isomorphic properties of their finite dimensional subspaces (such as type, cotype, super-reflexivity,\ldots) are preserved under coarse Lipschitz embeddings. This initiated what is now called the Ribe program, which aims at finding purely metric characterizations of local properties of Banach spaces.  We refer to \cite{Naor2012} and \cite{Ball2013} for a discussion of its origins and  motivations, and for a presentation of the most striking results in this direction.  In the last twenty years, the asymptotic structure of Banach spaces, which, very vaguely speaking, deals with the structure of their finite-codimensional subspaces or with the properties of weakly null sequences and trees, also proved to be central in the non-linear geometry of Banach spaces. We refer the reader to the seminal works of N. Kalton (\cite{Kalton2013TAMS} and \cite{Kalton2012MathAnn} for instance) and to the survey \cite{GLZ2014} and references therein. However, despite the accumulation of quite a few important stability results, there is no general analogue of Ribe's rigidity theorem in the setting of the asymptotic geometry of Banach spaces. Each new result requires an ad hoc argument. One of the main applications of this work is to exhibit two new properties related to the asymptotic smoothness of Banach spaces, that are stable under coarse Lipschitz equivalences. 
	
	\smallskip
	In this article, we describe in details four different properties dealing with the asymptotic uniform smoothness of Banach spaces, that we shall denote $\textsf{T}_p,\textsf{A}_p, \textsf{N}_p$, and $\textsf{P}_p$. In Section 2, we give their definitions in terms of two-players games on a Banach space $X$. Let us just mention for this introduction that each of them coincides with the existence of a special  form of so-called upper $\ell_p$ estimates for weakly null trees in $X$.  
	
	\smallskip 
	In Section 3, we give the main characterizations of  these properties, which involve, upper $\ell_p$ estimates for weakly null trees, the existence of quantitatively good equivalent asymptotically uniformly smooth norms and, dualy, the behaviour of the Szlenk index. We refer the reader to Section 3 for the precise statements and definitions. However, for the purpose of this introduction, we anticipate two definitions.
	
	Let us first recall the definition of the Szlenk derivation.  For a Banach space $X$, $K\subset X^*$ weak$^*$-compact, and $\ee>0$, we let $s_\ee(K)$ denote the set of $x^*\in K$ such that for each weak$^*$-neighborhood $V$ of $x^*$, $\text{diam}(V\cap K)\ge \ee$. For $1\leqslant q<\infty$, we say $X$ has $q$-\emph{summable Szlenk index} provided there exists a constant $c>0$ such that for any $n\in\nn$ and any $\ee_1, \ldots, \ee_n \geqslant 0$ such that $s_{\ee_1}\ldots s_{\ee_n}(B_{X^*})\neq \varnothing$, $\sum_{i=1}^n \ee_i^q \leqslant c^q$.   In the $q=1$ case, we refer to this as \emph{summable Szlenk index} rather than $1$-summable Szlenk index. 
	
	We also define the \emph{modulus of asymptotic uniform smoothness} of $X$. If $X$ is infinite-dimensional, for $\sigma\geqslant 0$,  we define 
	\[\overline{\rho}_X(\sigma) = \underset{y\in B_X}{\sup}\ \underset{E\in \co(X)}{\inf}\  \underset{x\in B_E}{\sup} \|y+\sigma x\|-1,\] 
	where $\co(X)$ denotes the set of finite codimensional subspaces of $X$.
	
	We assume, for this introduction, that what we mean by weakly null trees is understandable and refer to Section 3 for a detailed discussion. We can now state our main new result from Section 3. It describes various characterizations of our property $\textsf{A}_p$.  The renorming characterization is completely new and will be crucial for our non-linear stability results.

	\begin{thmAlfa} 
		Fix $1<p<\infty$ and let $q$ be conjugate to $p$. Let $X$ be a Banach space. The following are equivalent  
		\begin{enumerate}[(i)]
			\item $X\in \textsf{\emph{A}}_p$. 
			\item There exists a constant $c>0$ such that for any weak neighborhood base $D$ at $0$ in $X$, any $n\in\nn$,  and any weakly null tree $(x_t)_{t\in D^{\leqslant n}}$ in the unit ball of $X$, there exists $t\in D^n$ such that 
			
			$$\forall a=(a_i)_{i=1}^n \in \ell_p^n,\ \ \Big\|\sum_{i=1}^n a_ix_{t|_i}\Big\|\le c\|a\|_p.$$ 
			\item There exists a constant $M\ge 1$ and a constant $C>0$, such that for any $\tau \in (0,1]$ there exists a norm $|\ |$ on $X$ satisfying $M^{-1}\|x\|_X\le |x|\le M\|x\|_X$ for all $x\in X$ and
			$$\forall \sigma\ge \tau,\ \ \overline{\rho}_{|\ |}(s)\le Cs^p.$$
			\item $X$ has $q$-summable Szlenk index.   
		\end{enumerate}
	\end{thmAlfa}
	
	\smallskip 
	Let us recall that a property $(P)$ of Banach spaces is separably determined if a Banach space $X$ has $(P)$ if and only if all its separable subspaces have $(P)$. In Section 4, we provide a short and unified proof of the fact that all these properties are separably determined.
	
	\smallskip 
	Section 5 is devoted to the study of the three-space problem for $\textsf{T}_p,\textsf{A}_p, \textsf{N}_p$ and $\textsf{P}_p$. A property $(P)$ of Banach spaces is a \emph{Three-Space Property} (3SP in short) if it passes to quotients and subspaces and a Banach space $X$ has $(P)$ whenever it admits a subspace $Y$ such that $Y$ and $X/Y$ have $(P)$. The properties $\textsf{T}_p,\textsf{A}_p, \textsf{N}_p$ and $\textsf{P}_p$ pass quite simply to subspaces, quotients or isomorphs and it was proved in \cite{DKC} that $\textsf{P}_p$ is a 3SP. First, we take the opportunity of this paper to provide a more direct argument for this. Then, with a single example, we show  
	
	\begin{thmAlfa}
		Let $p\in (1,\infty)$. Then $\textsf{\emph{T}}_p,\textsf{\emph{A}}_p$, and $\textsf{\emph{N}}_p$ are not three-space properties.
	\end{thmAlfa}
	
	Finally, and this is the main result of this section, we show
	
	\begin{thmAlfa} Asymptotic uniform flattenability (property $\textsf{\emph{T}}_\infty$) and summable Szlenk index (property $\textsf{\emph{A}}_\infty$) are three-space properties.
	\end{thmAlfa}
	
	\smallskip 
	A net in a metric space $(M,d)$ is a subset $\mathcal M$ of $M$ such  that there exist $0<a<b$ so that for every $z\neq z'$ in $\cal M$,  $d(z,z')\ge a$ and for
	every $x$ in $M$, $d(x,\cal M)< b$. Let us use, for the simplicity of this introduction, that two infinite dimensional Banach spaces $X$ and $Y$ are coarse Lipschitz equivalent if and only if there exist two nets in $X$ and $Y$ that are Lipschitz equivalent. The precise definition of a coarse Lipschitz equivalence is given in Section 6 and can be roughly described as a Lipschitz equivalence at large distances. In \cite{GKL2000} and \cite{GKL2001}, it is proved that $\textsf{T}_p$ is stable under Lipschitz equivalences, $\textsf{P}_p$ is stable under coarse Lipschitz equivalences and $\textsf{A}_\infty=\textsf{N}_\infty$ is stable under coarse Lipschitz equivalences. In \cite{Kalton2013}, N. Kalton proved that for $1<p<\infty$, the class $\textsf{T}_p$ is not stable under coarse Lipschitz equivalences. Thanks to our renorming theorems from Section 3, we can almost complete this set of results.
	
	\begin{thmAlfa} 
		Let $p\in (1,\infty)$. Then, the class $\textsf{\emph{A}}_p$ and the class $\textsf{\emph{N}}_p$ are stable under coarse Lipschitz equivalences.
	\end{thmAlfa}
	Let us point that for $\textsf{N}_p$, this is deduced from already existing results, but was unnoticed,  while for $\textsf{A}_p$ it relies on our new renorming characterization. 
	
	\smallskip 
	In Section 7 we conclude this paper by gathering a few known examples of $\textsf{T}_\infty$ or $\textsf{A}_\infty$ spaces and related questions.

	\section{The properties}
	
	All Banach spaces are over the field $\mathbb{K}$, which is either $\mathbb{R}$ or $\mathbb{C}$.   We let $B_X$ (resp. $S_X$) denote the closed unit ball (resp. sphere) of $X$.    By \emph{subspace}, we shall always mean closed subspace.   Unless otherwise specified, all spaces are assumed to be infinite-dimensional. Throughout, we let $\textsf{Ban}$ denote the class of all Banach spaces over $\mathbb{K}$.  We let $\textsf{Sep}$ denote the class of separable members of $\textsf{Ban}$. 
	
	We now recall the definition of the Szlenk index, based on the Szlenk derivation introduced in Section 1.  For a Banach space $X$, $K\subset X^*$ weak$^*$-compact, and $\ee>0$, we  define the transfinite derivations 
	\[s_\ee^0(K)=K,\] \[s^{\xi+1}_\ee(K)=s_\ee(s^\xi_\ee(K)),\] and if $\xi$ is a limit ordinal, \[s^\xi_\ee(K)=\bigcap_{\zeta<\xi}s_\ee^\zeta(K).\]  
	For convenience, we let $s_0(K)=K$.     If there exists an ordinal $\xi$ such that $s^\xi_\ee(K)=\varnothing$, we let $Sz(K,\ee)$ denote the minimum such ordinal, and otherwise we write $Sz(K,\ee)=\infty$.   We let $Sz(K)=\sup_{\ee>0} Sz(K,\ee)$, where $Sz(K)=\infty$ if $Sz(K,\ee)=\infty$ for some $\ee>0$.   We let $Sz(X,\ee)=Sz(B_{X^*},\ee)$ and $Sz(X)=Sz(B_{X^*})$. In this work, we will exclusively be concerned with Banach spaces $X$ such that $Sz(X)\leqslant \omega$, where $\omega$ is the first infinite ordinal.   Since $Sz(X)=1$ if and only if $X$ has finite dimension, and otherwise $Sz(X)\geqslant \omega$, we will actually only be concerned with the case $Sz(X)=\omega$.  By compactness, $Sz(X)\leqslant \omega$ if and only if $Sz(X,\ee)$ is a natural number for each $\ee>0$.   We note that $Sz(X)<\infty$ if and only if $X$ is \emph{Asplund}. One characterization of Asplund spaces is that every separable subspace has a separable dual. 
	
	We recall that for any Banach space $X$ and $0<\ee, \delta<1$, 
	\[Sz(X, \ee\delta) \leqslant Sz(X, \ee)Sz(X, \delta).\] 
	From this it follows that if $Sz(X,\ee)$ is a natural number for each $\ee>0$, the Szlenk power type \[\textsf{p}(X):= \underset{\ee\to 0^+}{\lim} \frac{\log Sz(X,\ee)}{|\log(\ee)|}\] is finite.  It also holds that for any ordinal $\xi$, any $\ee>0$, and any natural number $n$, if $Sz(X,\ee)>\xi$, then $Sz(X,\frac{\ee}{n})> \xi n$.  Indeed, this follows from realizing \[B_{X^*} = \frac{1}{n}B_{X^*}+\ldots + \frac{1}{n}B_{X^*}\] and noting that the $\frac{\ee}{n}$-derivations act on one summand at a time in the same way that the $\ee$-derivations act on $B_{X^*}$.   Therefore the $\ee$-Szlenk index grows subgeometrically but superarithmetically.   The superarithmetic growth implies that for any infinite dimensional Banach space, $\textsf{p}(X)\geqslant 1$. 
	
	We define the following modulus. For $\eps>0$,
	$$\overline{\theta}^*_X(\eps)= \sup\{\delta \ge 0,\ s_\eps(B_{X^*})\subset (1-\delta)B_{X^*}\}.$$


	We recall again the definition of the \emph{modulus of asymptotic uniform smoothness} of $X$. If $X$ is infinite-dimensional, for $\sigma\geqslant 0$,  we define 
	\[\overline{\rho}_X(\sigma) = \underset{y\in B_X}{\sup}\ \underset{E\in \co(X)}{\inf}\  \underset{x\in B_E}{\sup} \|y+\sigma x\|-1,\] 
	where $\co(X)$ denotes the set of finite codimensional subspaces of $X$. For the sake of completeness, we define $\overline{\rho}_X(\sigma)=0$ for all $\sigma\geqslant 0$, when $X$ is finite-dimensional.    We note that 
	\[\overline{\rho}_X(\sigma) = \sup\{\underset{\lambda}{\lim\sup}\|y+\sigma x_\lambda\|-1:(x_\lambda)\subset B_X\text{\ is a weakly null net}\}.\]  
	It follows easily from the triangle inequality that $\overline{\rho}_X$ is a convex function. Since $\overline{\rho}_X(0)=0$, we deduce that $\sigma\mapsto \frac{\overline{\rho}_X(\sigma)}{\sigma}$ is non-decreasing on $(0,\infty)$.  Therefore \[\underset{\sigma>0}{\inf} \frac{\overline{\rho}_X(\sigma)}{\sigma} = \underset{\sigma\to 0^+}{\ \lim\sup\ } \frac{\overline{\rho}_X(\sigma)}{\sigma}.\]   
	We say $X$ is \emph{asymptotically uniformly smooth} (in short \emph{AUS}) if 
	$$\inf_{\sigma>0} \frac{\overline{\rho}_X(\sigma)}{\sigma}= 0.$$   
	We say $X$ is \emph{asymptotically uniformly smoothable} (\emph{AUS-able}) if $X$ admits an equivalent AUS norm.     For $1<p<\infty$, we say $X$ is $p$-\emph{asymptotically uniformly smooth} (in short $p$-\emph{AUS}) if 
	\[\sup_{\sigma>0} \frac{\overline{\rho}_X(\sigma)}{\sigma^p}<\infty.\]  
	We say $X$ is $p$-\emph{asymptotically uniformly smoothable} ($p$-\emph{AUS-able}) if $X$ admits an equivalent $p$-AUS norm.  We say $X$ is \emph{asymptotically uniformly flat} (\emph{AUF}) if there exists $\sigma_0>0$ such that $\overline{\rho}_X(\sigma_0)=0$.  We say $X$ is \emph{asymptotically uniformly flattenable} if $X$ admits an equivalent AUF norm. Of course, $p$-AUS spaces and AUF spaces are AUS spaces. 
	
	It is well known that the dual Young function of the modulus of asymptotic uniform smoothness is equivalent to the so-called modulus of weak$^*$ asymptotic uniform convexity $\overline{\delta}^*_X$ (see Proposition 2.1 and Corollary 2.3 in \cite{DKLR2017}). It is also clear that $\overline{\delta}^*_X$ is equivalent to $\overline{\theta}^*_X$. It will be more convenient for us to work with $\overline{\theta}^*_X$. We shall only need the following version of Proposition 2.1 in \cite{DKLR2017}.
	
	\begin{proposition}\label{moduliduality} 
		There exists a universal constant $C\ge 1$ such that for any Banach space $X$ and any $0<\sigma,\tau<1$,
		\begin{enumerate}
			\item If $\overline{\rho}_X(\sigma)<\sigma \tau$, then $\overline{\theta}^*_X(C\tau)\ge \sigma\tau$.
			\item If $\overline{\theta}^*_X(\tau)>\sigma \tau$, then $\overline{\rho}_X(\frac{\sigma}{C}) \le \sigma\tau$.
		\end{enumerate}
		
	\end{proposition}

	For $1\leqslant q<\infty$, a Banach space $X$, and a sequence $(x_i)_{i=1}^\infty\subset X$, we define the (possibly infinite) quantity \[\|(x_i)_{i=1}^\infty\|_q^w = \sup\{\|(x^*(x_i))_{i=1}^\infty\|_{\ell_q}: x^*\in B_{X^*}\}.\]   We also define this quantity for finite sequences, \[\|(x_i)_{i=1}^n\|_q^w = \sup\{\|(x^*(x_i))_{i=1}^n\|_{\ell_q^n}: x^*\in B_{X^*}\}.\] 
	Note that, if $p\in (1,\infty]$ is the conjugate exponent of $q$, we have that 
	$$\|(x_i)_{i=1}^\infty\|_q^w = \inf\big\{c\in (0,\infty],\ \forall a=(a_i)_{i=1}^\infty \in \ell_p\ \|\sum_{i=1}^\infty a_ix_i\|\le c\|a\|_p\big\}.$$
	A similar formula is valid for $\|(x_i)_{i=1}^n\|_q^w$. 
	
	\medskip
	We next define four different two-players games on a Banach space $X$. Fix $1<p\leqslant \infty$ and let $1/p+1/q=1$.   For $c>0$ and $n\in\nn$, we define the $T(c,p)$ game on $X$, the $A(c,p,n)$ game, and the $N(c,p,n)$ game. Let $D$ be a weak neighborhood base at $0$ in $X$.  In the $T(c,p)$ game, Player I chooses a weak neighborhood $U_1\in D$, and Player II chooses $x_1\in U_1\cap B_X$. Player I chooses $U_2\in D$, and Player II chooses $x_2\in U_2\cap B_X$.  Play continues in this way until $(x_i)_{i=1}^\infty$ has been chosen. Player I wins if $\|(x_i)_{i=1}^\infty\|_q^w \leqslant c$, and Player II wins otherwise.  
	
	The $A(c,p,n)$ game is similar, except the game terminates after the $n^{th}$ turn. Player I wins if $\|(x_i)_{i=1}^n\|_q^w\leqslant c$, and Player II wins otherwise.  
	
	In the $N(c,p,n)$ game, as in the $A(c,p,n)$ game, the game terminates after the $n^{th}$ turn. Player I wins if $\Bigl\|\sum_{i=1}^n x_i\Bigr\|\leqslant c n^{1/p},$ and Player II wins otherwise. 
	
	Finally, in the $\Theta(c,n)$ game, Player I wins if $\Bigl\|\sum_{i=1}^n x_i\Bigr\|\leqslant c,$ and Player II wins otherwise. 
	
	It is known (see \cite{CauseyPositivity2018}, Section 3) that each of these games is determined. That is, in each game, either Player I or Player II has a winning strategy.  We let $\textsf{t}_p(X)$ denote the infimum of $c>0$ such that Player I has a winning strategy in the $T(c,p)$ game, provided such a $c$ exists, and we let $\textsf{t}_p(X)=\infty$ otherwise.  We let $\textsf{a}_{p,n}(X)$ denote the infimum of $c>0$ such that Player II has a winning strategy in the $A(c,p,n)$ game, and we let $\textsf{a}_p(X)=\sup_n \textsf{a}_{p,n}(X)$. We note that $\textsf{a}_p(X)$ is the infimum of $c>0$ such that for each $n\in\nn$, Player I has a winning strategy in the $A(c,p,n)$ game if such a $c$ exists, and $\textsf{a}_p(X)=\infty$ otherwise. We let $\theta_n(X)$ denote the infimum of $c>0$ such that Player I has a winning strategy in the $\Theta(c,n)$ game, noting that $\theta_n(X)\leqslant n$.  Finally, we let $\textsf{n}_{p,n}(X)= \theta_n(X)/n^{1/p}$ and $\textsf{n}_p(X)=\sup_n \textsf{n}_{p,n}(X)$, noting that $\textsf{n}_p(X)$ is the infimum of $c>0$ such that for each $n\in\nn$, Player I has a winning strategy in the $N(c,p,n)$ game, provided such a $c$ exists, and $\textsf{n}_p(X)=\infty$ otherwise. 
	
	\begin{remark}
		The existence of winning strategies, and therefore the constants associated with these games, do not depend upon the particular choice of the weak neighborhhod base $D$. Therefore in the case that $X$ has a separable dual, these constants are sequentially determined. Let us indicate the argument.
	\end{remark}
	\begin{proof} If $D_1, D_2$ are two weak neighborhood bases at $0$ in $X$, and Player I has a winning strategy in any of the games above when Player I is required to choose from $D_1$, then this winning strategy can be used to construct a winning strategy choosing from $D_2$ by choosing at each stage of the game any member of $D_2$ which is a subset of the member of $D_1$ indicated by the winning strategy. From this it follows that  the values of the associated constants also do not depend on $D$.  In particular, in the case that $X^*$ is separable, we can use a fixed countable, linearly ordered weak neighborhood base $D$.
	\end{proof}
	
	Let $D^{\leqslant n}=\cup_{i=1}^n D^i$. Let $D^{<\omega}=\cup_{i=1}^\infty D^i$, and let $D^\omega$ denote the set of all infinite sequences whose members lie in $D$. Let $D^{\leqslant \omega}=D^{<\omega}\cup D^\omega$.  For $s,t\in D^{<\omega}$, we let $s\smallfrown t$ denote the concatenation of $s$ with $t$. We let $|t|$ denote the length of $t$.  For $0\leqslant i\leqslant|t|$, we let $t|_i$ denote the initial segment of $t$ having length $i$, where $t|_0=\varnothing$ is the empty sequence.  If $s\in \{\varnothing\}\cup D^{<\omega}$, we let $s\prec t$ denote the relation that $s$ is a proper initial segment of $t$.
	
	We say a function $\varphi:D^{<\omega}\to D^{<\omega}$ is a \emph{pruning} provided that \begin{enumerate}[(i)]\item $|\varphi(t)|=|t|$ for all $t\in D^{\leqslant n}$, \item if $s\prec t$, then $\varphi(s)\prec \varphi(t)$, \item if $\varphi((U_1, \ldots, U_k))=(V_1, \ldots, V_k)$, then $V_k\subset U_k$. \end{enumerate} We define prunings $\varphi:D^{\leqslant n}\to D^{\leqslant n}$ similarly.  
	

	Given $D$ a weak neighborhood base of $0$ in $X$ and $(x_t)_{t\in D^{<\omega}}\subset X$, we say $(x_t)_{t\in D^{<\omega}}$ is 
	\begin{enumerate}[(i)]
		\item \emph{weakly null of type} I provided that for each $t=(U_1, \ldots, U_k)$, $x_t\in U_k$, 
		\item \emph{weakly null of type} II provided that for each $t\in \{\varnothing\}\cup D^{<\omega}$,\\  $(x_{t\smallfrown (U)})_{U\in D}$ is a weakly null net. Here $D$ is directed by reverse inclusion. 
	\end{enumerate} 
	The notions of weakly null of types I and II for collections indexed by $D^{\leqslant n}$ are defined similarly. Note that a weakly null collection of type I is weakly null of type II. We now link these notions with our various games.
	
	\begin{proposition} Let $X$ be a Banach space, let $p\in (1,\infty]$, and $c>0$. Then, 
		Player II has a winning strategy in the $T(c,p)$ game on $X$ if and only if there exists a collection $(x_t)_{t\in D^{<\omega}}\subset B_X$ such that 
		\begin{enumerate}[(a)]
			\item $(x_t)_{t\in D^{<\omega}}$ is weakly null of type I, and 
			\item for each $\tau\in D^\omega$, $\|(x_{\tau|_i})_{i=1}^\infty\|_q^w > c$. 
		\end{enumerate}
	\end{proposition}
	
	\begin{proof} First, if such a collection exists, we can use it to define a winning strategy for Player II in the $T(c,p)$ game. When Player I chooses $U_1\in D$, then Player II chooses $x_{(U_1)}$. Player I chooses $U_2\in D$, to which Player II's response is $x_{(U_1, U_2)}$. Play continues in this way, and the result is $(x_{\tau|_i})_{i=1}^\infty$ for some $\tau\in D^\omega$, which satisfies $\|(x_{\tau|_i})_{i=1}^\infty\|_q^w>c$.  On the other hand, if Player II has a winning strategy in the $T(c,p)$ game, we define by induction on $k$ the vector $x_{(U_1, \ldots, U_k)}$ to be Player II's response according to this winning strategy following the choices $U_1, x_{(U_1)}, U_2, x_{(U_1, U_2)}, \ldots, U_k$ in the $T(c,p)$ game.   It follows from the rules of the game that (a) is satisfied, and it follows from the fact that Player II plays according to a winning strategy that (b) is satisfied.
	\end{proof}
	
	Analogous statements about $(x_t)_{t\in D^{\leqslant n}}\subset B_X$ can be made for the $A(c,p,n)$, $N(c,p,n)$, and $\Theta(c,n)$ games.   We also have
	
	\begin{proposition} Let $X$ be a Banach space, let $p\in (1,\infty]$, and $c>0$. Then, Player II has a winning strategy in the $T(c,p)$ game if and only if there exists a collection $(x_t)_{t\in D^{<\omega}}\subset B_X$ such that \begin{enumerate}[(a)]
			\item $(x_t)_{t\in D^{<\omega}}$ is weakly null of type II, and 
			\item for each $\tau\in D^\omega$, $\|(x_{\tau|_i})_{i=1}^\infty\|_q^w > c$. 
		\end{enumerate}
	\end{proposition}
	
	\begin{proof} Since any collection which is weakly null of type I is also weakly null of type II, by the previous proposition,  it is sufficient to note that if $(x_t)_{t\in D^{<\omega}}\subset B_X$ is weakly null of type II, then there exists a pruning $\varphi:D^{<\omega}\to D^{<\omega}$ such that $(x_{\varphi(t)})_{t\in D^{<\omega}}\subset B_X$ is weakly null of type I. Moreover, property (b) is retained by the collection $(x_{\varphi(t)})_{t\in D^{<\omega}}$.
	\end{proof}
	
	Again, analogous statements hold for collections indexed by $D^{\leqslant n}$ and the games $A(c,p,n)$, $N(c,p,n)$, and $\Theta(c,n)$. Unless otherwise specified, by a weakly null collection $(x_t)_{t\in D^{<\omega}}$ in $X$, we shall mean weakly null of type II. However, it might be convenient to use that we may assume it to be weakly null of type I. 
	
	\begin{remark} As we already mentioned, in the case that $X^*$ is separable, we can use a fixed countable, linearly ordered weak neighborhood base $D$ and, by identifying $D$ with $\nn$, characterize the constants $\textsf{t}_p(X), \textsf{a}_{p,n}(X), \textsf{a}_p(X),$ $ \theta_n(X), \textsf{n}_{p,n}(X), \textsf{n}_p(X)$ using trees indexed by $\nn^{<\omega}$ or $\nn^{\leqslant n}$ rather than $D^{<\omega}$ or $D^{\leqslant n}$. Now, if $X$ is separable and there exists some $1<p\leqslant \infty$ such that any of the constants $\textsf{t}_p(X), \textsf{a}_p(X), \textsf{n}_p(X)$, then $Sz(X)\leqslant \omega$, $X$ is Asplund, and $X^*$ is separable. However, we can only use $\nn^{<\omega}$ (resp. $\nn^{\leqslant n}$) in place of $D^{<\omega}$ (resp. $D^{\leqslant n}$) as index sets to compute the values of these constants  once we already know that $X^*$ is separable, because of examples like $\ell_1$ with the Schur property.  So, for example, once we know $\textsf{p}(X)$ is finite, we can characterize its value using trees indexed by $\nn^{<\omega}$, but we cannot use trees indexed by $\nn^{<\omega}$ to determine whether $\textsf{t}_p(X)$ is finite.	
	\end{remark}
	
	We conclude this section with elementary statements that we shall use to stabilize weakly null trees. Note that the operation described below is actually a pruning.
	
	\begin{proposition}\label{prune1}  Let $(D,\leqslant_D)$ be any directed set, $F$ a finite set, $n$ a natural number, and $f:D^n\to F$ a function.  There exists $\theta:D^{\leqslant n}\to D^{\leqslant n}$ preserving lengths and initial segments such that 
		\begin{enumerate}[(i)]
			\item if $\theta((U_1, \ldots, U_k))=(V_1, \ldots, V_k)$, then $U_k\leqslant_D V_k$, 
			\item $f\circ \theta|_{D^n}$ is constant. 
		\end{enumerate}
	\end{proposition}
	
	\begin{proof} We work by induction. For $x\in F$, let $I_x=\{U\in D:f((U))=x\}$. Since $\cup_{x\in F}I_x=D$ and $F$ is finite, there exists some $x\in F$ such that $I_x$ is cofinal in $D$. This means that for any $U\in D$, there exists $V_U\in I_x$ such that $U\leqslant_D V_U$. Define $\theta((U))=(V_U)$ and note that $f\circ \theta|_{D^1}\equiv x$. 
		
		Next, assume the result holds for some $n$ and fix $f:D^{n+1}\to F$.  For each $U\in D$, define $f_U:D^n\to F$ by $f_U((U_1, \ldots, U_n))=f((U, U_1, \ldots, U_n))$.  By the inductive hypothesis, there exist $\theta_U:D^{\leqslant n}\to D^{\leqslant n}$ which preserves lengths and initial segments and satisfying $(i)$ and $(ii)$. Fix $x_U\in F$ such that $f_U\circ \theta_U|_{D^n}\equiv x_U$.   Define $g:D^1\to F$ by $g((U))=x_U$. By the base case, there exists $\phi:D^1\to D^1$ satisfying $(i)$ and $(ii)$ with $f$ replaced by $g$.  Define $\theta:D^{\leqslant n+1}\to D^{\leqslant n+1}$ by $\theta((U))=\phi(U)$ and $\theta((U, U_1, \ldots, U_k)) = \phi(U)\smallfrown \theta_{\phi(U)}(U_1, \ldots, U_k)$. 
		
	\end{proof}
	
	\begin{corollary}\label{prune2} Let $(D,\leqslant_D)$ be any directed set, $(K,d)$ a totally bounded metric space, $n$ a natural number, and $f:D^n\to K$ a function.  For any $\ee>0$, there exist $\theta:D^{\leqslant n}\to D^{\leqslant n}$ preserving lengths and initial segments and a subset $B$ of $K$ of diameter less than $\ee$ such that 
		\begin{enumerate}[(i)]
			\item if $\theta((U_1, \ldots, U_k))=(V_1, \ldots, V_k)$, then $U_k\leqslant_D V_k$, 
			\item $f(\theta(t))\in B$ for all $t\in D^n$. 
		\end{enumerate}
	\end{corollary}
	
	\begin{proof} Let $B_1, \ldots, B_m$ be a cover of $K$ by sets of diameter less than $\ee$. Define $g:D^n\to \{1, \ldots, m\}$ by letting $g(t)=\min \{i\leqslant m: f(t)\in B_i\}$.  Apply Proposition \ref{prune1} to $g$. 
		
	\end{proof}
	
	


	\section{The properties and their relations}\label{relations}
	
	For $1<p\leqslant \infty$, we let $\textsf{T}_p$ denote the class of Banach spaces $X$ such that $\textsf{t}_p(X)<\infty$.   The classes $\textsf{A}_p$ and $\textsf{N}_p$ are defined similarly using $\textsf{a}_p$ and $\textsf{n}_p$.  We let $\textsf{P}_p=\bigcap_{1<r<p}\textsf{T}_r$.  We let $\textsf{D}_1$ denote the class of all Banach spaces the Szlenk index of which does not exceed $\omega$. We now present the following alternative descriptions of each class. We have chosen to quickly indicate the easy arguments, to give precise references for others and to detail the new ones. We give this overview, insisting on the renorming characterizations, as they are crucial for our non-linear applications. 
	
	\medskip We start with the description of $\textsf{{T}}_p$. The next theorem is the main result from \cite{CauseyPositivity2018}. We briefly explain the easy implications and emphasize the key part of the proof. 
	
	\begin{theorem}\label{ttheorem} Fix $1<p\leqslant \infty$ and let $q$ be conjugate to $p$. Let $X$ be a Banach space. The following are equivalent  
		\begin{enumerate}[(i)]
			\item $X\in \textsf{\emph{T}}_p$. 
			\item There exists a constant $c>0$ such that for any weak neighborhood base $D$ at $0$ in $X$ and any weakly null $(x_t)_{t\in D^{<\omega}}\subset B_X$, there exists $\tau\in D^\omega$ such that $\|(x_{\tau|_i})_{i=1}^\infty\|_q^w \leqslant c$. 
			\item $X$ is $p$-AUS-able (resp. AUF-able if $p=\infty$). 
			\item There exist an equivalent norm $|\cdot|$ on $X$ and $c>0$ such that for each $\ee \in [0,1]$, $s_\ee(B_{X^*}^{|\cdot|})\subset (1-c\ee^q)B_{X^*}^{|\cdot|}$. In other words, $\overline{\theta}^*_{|\cdot|}(\eps)\ge c\eps^q$
		\end{enumerate}
	\end{theorem}
	
	\begin{proof} The equivalence between  $(i)$ and $(ii)$ follows immediately from our discussions on winning strategies in the $T(c,p)$ game.  More precisely, $\textsf{t}_p(X)$ is the infimum of those $c$ for which $(ii)$ holds. 

		\smallskip
		The equivalence between $(iii)$ and $(iv)$ is a immediate consequence of the duality Proposition \ref{moduliduality}.
		
		\smallskip
		We now detail the rather simple implication $(iii)\Rightarrow (i)$ and assume, as we may, that $X$ is $p$-AUS. We look at $1<p<\infty$ and $p=\infty$ separately. First consider the case $1<p<\infty$.  We note that $\sup_{\sigma>0}\overline{\rho}_X(\sigma)/\sigma^p<\infty$ if and only if there exists a constant $C\geqslant 1$ such that for any $x\in X$ and $\sigma\geqslant 0$, there exists a weak neighborhood $U$ of $0$ in $X$ such that for any $y\in U\cap B_X$, $\|x+\sigma y\|^p \leqslant \|x\|^p + C^p\sigma^p+\ee$.  A finite net argument yields that for any compact $G\subset X$ and $\ee>0$, there exists a weak neighborhood $U$ of $0$ such that for any $x\in G$, any scalar $b$ with $|b|\leqslant 1$, and any $y\in U\cap B_X$, $\|x+b y\|^p \leqslant \|x\|^p + C^p|b|^p$.  Using this fact, for $\ee>0$, we can define a winning strategy for Player I in the $T(C+\ee,p)$ game by fixing $(\ee_i)_{i=1}^\infty\subset (0,1)$.  Player I's initial choice $U_1$ is arbitrary. Once $U_1, x_1, \ldots, U_n, x_n$ have been chosen, let 
		$$G=\Bigl\{\sum_{i=1}^n b_i x_i: (b_i)_{i=1}^n \in B_{\ell_p^n}\Bigr\}$$ 
		and choose $U_{n+1}$ such that for any $x\in G$, $y\in U_{n+1}\cap B_X$, and any $b$ with $|b|\leqslant 1$, $\|x+b y \|^p \leqslant \|x\|^p + C^p|b|^p+\ee_{n+1}$.   This completes the recursive construction.   For any $m\in\nn$ and $(b_i)_{i=1}^m\in B_{\ell_p^m}$, 
		\begin{align*} \Bigl\|\sum_{i=1}^m b_i x_i\Bigr\|^p & \leqslant \Bigl\|\sum_{i=1}^{m-1}b_ix_i\Bigr\|^p + C^p|b_m|^p+\ee_m \\ & \leqslant \Bigl\|\sum_{i=1}^{m-2}b_ix_i\Bigr\|^p + C^p|b_{m-1}|^p + C^p|b_m|^p + \ee_{m-1}+\ee_m \\ & \leqslant C^p\sum_{i=1}^m |b_i|^p + \sum_{i=1}^m \ee_i \leqslant C^p+\sum_{i=1}^\infty \ee_i.
		\end{align*}  
		If  $\sum_{i=1}^\infty \ee_i$ was chosen small enough, depending on the modulus of continuity of the function $t\mapsto t^p$ on $[0,C+1]$, this strategy is a winning strategy for Player I in the $T(C+\ee,p)$ game. Therefore $X$ has $\textsf{T}_p$. For the $p=\infty$ case, the argument is similar, except there exists a constant $C$ such that for any $x\in X$ and $\sigma >0$, there exists a weak neighborhood $U$ of $0$ such that for any $y\in U\cap B_X$, $\|x+\sigma y\|\leqslant \max\{\|x\|, C\sigma\}$.  
		
		\smallskip 
		Finally, we refer the reader to \cite{CauseyPositivity2018} for the difficult implication $(i)\Rightarrow (iv)$.

	\end{proof}

	We now turn to the characterizations of $\textsf{A}_p$. Note that item $(iii)$ is a completely new characterization. For that reason we recall the old arguments and detail the new ones. As we will see later, $\textsf{A}_\infty=\textsf{N}_\infty$, so we limit ourselves to $p\in (1,\infty)$ in the next statement. 
	\begin{theorem}\label{atheorem} 
		Fix $1<p<\infty$ and let $q$ be conjugate to $p$. Let $X$ be a Banach space. The following are equivalent  
		\begin{enumerate}[(i)]
			\item $X\in \textsf{\emph{A}}_p$. 
			\item There exists a constant $c>0$ such that for any weak neighborhood base $D$ at $0$ in $X$, any $n\in\nn$,  and any weakly null $(x_t)_{t\in D^{\leqslant n}}\subset B_X$, there exists $t\in D^n$ such that $\|(x_{t|_i})_{i=1}^n\|_q^w\le c$. 
			\item There exist a constant $M\ge 1$ and a constant $C>0$, such that for any $\tau \in (0,1]$ there exists a norm $|\ |$ on $X$ satisfying $M^{-1}\|x\|_X\le |x|\le M\|x\|_X$ for all $x\in X$ and
			$$\forall \sigma\ge \tau,\ \ \overline{\rho}_{|\ |}(s)\le Cs^p.$$
			\item $X$ has $q$-summable Szlenk index.   
		\end{enumerate}
	\end{theorem}
	
	\begin{proof} The equivalence between $(i)$ and $(ii)$ follows again from our initial discussion on games. 
		
		
		\smallskip
		The implication $(ii) \Rightarrow (iii)$ is new. Let us prove it. Fix $1<p<\infty$.  Suppose that $X$ is a Banach space and $a\geqslant 1$ is such that for each $n\in\nn$ and  $(x_t)_{t\in D^{\leqslant n}}\subset B_X$ weakly null, there exists $t\in D^n$ such that for all scalar sequences $(a_i)_{i=1}^n$, 
		\[\Bigl\|\sum_{i=1}^n a_ix_{t|_i}\Bigr\|^p\leqslant a^p\sum_{i=1}^n |a_i|^p.\]    
		
		We first note that for any $x\in X$, $n\in\nn$, and $(x_t)_{t\in D^{\leqslant n}}\subset B_X$ weakly null, there exists $t\in D^n$ such that for all scalar sequences $(a_i)_{i=1}^n$, 
		\begin{equation}\label{f1}
			\Bigl\|x+\sum_{i=1}^n a_ix_{t|_i}\Bigr\|^p\leqslant (2a)^p\Bigl[\|x\|^p+ \sum_{i=1}^n |a_i|^p\Bigr].
		\end{equation}
		Indeed, for an appropriate branch $t$, it holds that 
		\begin{align*} \Bigl\|x+\sum_{i=1}^n a_ix_{t|_i}\Bigr\|^p & \leqslant 2^p\max\Bigl\{\|x\|^p, \Bigl\|\sum_{i=1}^n a_ix_{t|_i}\Bigr\|^p\Bigr\}\\ & \leqslant (2a)^p\max\Bigl\{\|x\|^p,\sum_{i=1}^n |a_i|^p\Bigr\} \leqslant (2a)^p\Bigl[\|x\|^p+ \sum_{i=1}^n |a_i|^p\Bigr].
		\end{align*}
		
		Let now $A=2a$. Set $f_0(x)=\frac{\|x\|}{A}$ and for $n\in\nn$, define 
		\[f_n(x)=\Biggl[ \underset{(x_t)}{\ \sup\ }\underset{t}{\ \inf\ }\underset{(a_i)}{\ \sup\ } \frac{1}{A^p}\Bigl\|x+\sum_{i=1}^n a_ix_{t|_i}\Bigr\|^p-\sum_{i=1}^n |a_i|^p\Biggr]^{1/p},\] 
		where the outer supremum is taken over all weakly null collections $(x_t)_{t\in D^{\leqslant n}}$ in  $B_X$, the infimum is taken over $t\in D^n$, and the inner supremum is taken over all scalar sequences $(a_i)_{i=1}^n$. It follows from taking $x_t=0$ for all $t$ that $f_n(x)\geqslant \frac{\|x\|}{A}$ for all $x\in X$ and $n\in\nn$. On the other hand, it follows from  (\ref{f1}) that $f_n(x)\leqslant \|x\|$ for all $n\in\nn$. We also have that  $f_n(cx)=|c|f_n(x)$, for each $n\in\nn\cup \{0\}$, each $x\in X$, and each scalar $c$. Let us detail this last fact. It is clear that $f_n(0)=0$, so assume $c\neq 0$. It is also clear that $f_0(cx)=|c|f_n(x)$. Then we fix $n\in\nn$, $x\in X$ and $n\in\nn$. For an arbitrary $(x_t)_{t\in D^{\leqslant n}}\subset B_X$ weakly null and $b>f_n(x)$, there exists $t\in D^n$ such that for all $(a_i)_{i=1}^n$, \[\frac{1}{A^p}\Bigl\|x+\sum_{i=1}^n a_ix_{t|_i}\Bigr\|^p-\sum_{i=1}^n |a_i|^p \leqslant b^p.\]    Then 
		\begin{align*}
			&\frac{1}{A^p}\Bigl\|cx+\sum_{i=1}^n a_ix_{t|_i}\Bigr\|^p-\sum_{i=1}^n |a_i|^p\\ &=|c|^p\Bigl[\frac{1}{A^p}\Bigl\|x+\sum_{i=1}^n c^{-1}a_ix_{t|_i}\Bigr\|^p-\sum_{i=1}^n |c^{-1}a_i|^p  \Bigr] \leqslant |c|^pb^p.
		\end{align*}   
		Since this holds for any $(x_t)_{t\in D^{\leqslant n}}\subset B_X$ weakly null, it holds that $f_n(cx) \leqslant |c|f_n(x)$.   Repeating the argument, we deduce that $f_n(x) =f_n(c^{-1}cx) \leqslant |c|^{-1}f_n( cx)$, which gives the reverse inequality. 
		
		The key step will be to show the following:    for each $n\in\nn\cup \{0\}$, each weakly null net $(x_U)_{U\in D}\subset B_X$, and each $\sigma>0$, 
		\begin{equation}\label{f2}
			\underset{U}{\lim\sup} f_n(x+\sigma x_U)^p \leqslant f_{n+1}(x)^p+\sigma^p.
		\end{equation}
		So assume $\eta < \underset{U}{\lim\sup} f_n(x+\sigma x_U)^p$. By passing to a subnet and relabeling, we can assume $\eta < f_n(x+\sigma x_U)^p$ for all $U$.   For each $U$, we find $(x^U_t)_{t\in D^{\leqslant n}}\subset B_X$ weakly null such that for each $t\in D^n$, there exists $(a_i)_{i=1}^n$ such that \[\frac{1}{A^p}\Bigl\|x+\sigma x_U + \sum_{i=1}^n a_ix^U_{t|_i}\Bigr\|^p - \sum_{i=1}^n |a_i|^p > \eta.\]
		We define the weakly null collection $(x_t)_{t\in D^{\leqslant n+1}}\subset B_X$ by letting $x_{(U)}=x_U$ and $x_{(U, U_1, \ldots, U_k)}=x^U_{(U_1, \ldots, U_k)}$ for $1\leqslant k\leqslant n$.    By the definition of $f_{n+1}(x)^p$, for any $\ee>0$, there exists $s\in D^{n+1}$ such that for all $(b_i)_{i=1}^{n+1}$, \[\frac{1}{A^p}\Bigl\|x+\sum_{i=1}^{n+1} b_i x_{s|_i}\Bigr\|^p - \sum_{i=1}^{n+1}|b_i|^p \leqslant f_{n+1}(x)^p+\ee.\]   Write $s=(U, U_1, \ldots, U_n)$ and let $t=(U_1, \ldots, U_n)$.  Then there exists $(a_i)_{i=1}^n$ such that, combining this paragraph with the previous and letting $b_1=\sigma$ and  $b_{i+1}=a_i$ for $1\leqslant i\leqslant n$, it holds that  \begin{align*} \eta & < \frac{1}{A^p}\Bigl\|x+\sigma x_U + \sum_{i=1}^n a_i x^U_{s|_i}\Bigr\|^p -\sum_{i=1}^n |a_i|^p \\ & = \frac{1}{A^p}\Bigl\|x+\sum_{i=1}^{n+1}b_ix_{s|_i}\Bigr\|^p - \sum_{i=1}^{n+1}|b_i|^p +\sigma^p \leqslant f_{n+1}(x)^p+\ee +\sigma^p. \end{align*}
		Therefore $\eta < f_{n+1}(x)^p+\ee+\sigma^p$. Since $\eta < \underset{U}{\lim\sup}f_n(x+\sigma x_U)^p$ and $\ee>0$ were arbitrary, we have proved (\ref{f2}). 
		
		The next step is to average the $f_n^p$'s. So, fix $N\in\nn$ and define \[g_N(x)^p=\frac{1}{N}\sum_{n=0}^{N-1}f_n(x)^p.\]   
		Clearly, we still have that for all $x\in X$ and $N\in\nn$, $\frac{\|x\|}{A} \leqslant g_N(x) \leqslant \|x\|$ and $g_N(cx)=|c|g_N(x)$ for all scalars $c$. Then, applying (\ref{f2}) for each $n\in \{0,\ldots,N-1\}$, we obtain that for any weakly null net $(x_U)_{U\in D}\subset B_X$, any $N\in\nn$, and $x\in AB_X$, 
		\begin{equation}\label{f3}
			\underset{U}{\lim\sup}\, g_N(x+\sigma x_U)^p \leqslant g_N(x)^p+\sigma^p + \frac{A^p}{N}.
		\end{equation}
		
		The last stage of the proof is to ``convexify'' our function $g_N$. For that purpose, we set \[|x|_N = \inf\Bigl\{\sum_{i=1}^n g_N(x_i): n\in\nn, x=\sum_{i=1}^n x_i\Bigr\},\] which defines an equivalent norm on $X$ satisfying $\frac{\|x\|}{A}\leqslant |x|_N \leqslant \|x\|$. Moreover, $B_X^{|\cdot|_N}$ is the closed, convex hull of $\{x\in X: g_N(x)<1\}$. We shall now prove that 
		\begin{equation}\label{f4}
			\forall \sigma>0,\ \ \overline{\rho}_{(X, |\cdot|_N)}(\sigma) \leqslant \frac{A^p}{p}\big(\sigma^p + \frac{1}{N}\big).
		\end{equation}
		
		First we fix $y\in X$ such that $g_N(y)<1$. From this it follows that $\|y\|\leqslant A$.  Fix $\sigma>0$ and $(y_U)_{U\in D}\subset B_X^{|\cdot|_N}$ weakly null, define $x_U=A^{-1}y_U\in B_X$, so $(x_U)_{U\in D}\subset B_X$ is weakly null.  Then we apply (\ref{f3}) to get
		\begin{align*}
			\underset{U}{\lim\sup}\,|y+\sigma y_U|^p_N & \leqslant \underset{U}{\lim\sup}\, g_N(y+\sigma A x_U)^p\\ 
			&\leqslant  g_N(y)+\sigma^pA^p +\frac{A^p}{N}<1+\sigma^pA^p +\frac{A^p}{N}.
		\end{align*}
		Therefore, by concavity of the function  $h(t)=(1+t)^{1/p}$,
		\[\underset{U}{\lim\sup}\, |y+\sigma y_U|_N-1 \leqslant \big(1+ \sigma^p A^p + \frac{A^p}{N}\big)^{1/p}-1 \leqslant \frac{A^p}{p}\big(\sigma^p+\frac1N\big).
		\]
		Next fix $x\in B_X^{|\cdot|_N}$.  As noted above, $B_X^{|\cdot|_N}$ is the closed, convex hull of $\{y\in X:g_N(y)<1\}$.   Therefore for each $\ee>0$, we can find $y_1, \ldots, y_k\in X$ with $g_N(y_i)<1$ and convex coefficients $w_1, \ldots, w_k$ such that $|x-\sum_{i=1}^k w_iy_i|<\ee$.  Then 
		\begin{align*} 
			\underset{U}{\lim\sup}\, |x+ \sigma y_U| -1 & \leqslant \ee + \underset{U}{\lim\sup}\, \sum_{i=1}^k w_i (|y_i-\sigma y_U|_N-1)\\ 
			&\leqslant \ee + \sum_{i=1}^k w_i \Big(\frac{A^p}{p}\big(\sigma^p+\frac1N\big)\Big) = \ee + \frac{A^p}{p}\big(\sigma^p+\frac1N\big).
		\end{align*} 
		Since $\ee>0$ was arbitrary, this finishes the proof of (\ref{f4}). 
		
		Finally, it is clear, by taking $N$ large enough in (\ref{f4}), that for any $\tau>0$ there exists an equivalent norm $|\cdot|$ on $X$ such that $\frac{\|x\|}{A}\leqslant |x|\leqslant \|x\|$ and for any $\sigma\geqslant \tau$, $\overline{\rho}_{(X, |\cdot|)}(\sigma) \leqslant \frac{A_1^p}{p}\sigma^p$. We have proved that $X$ satisfies (iii).
		
		\smallskip
		Next we prove $(iii) \Rightarrow (iv)$, which is also new. So assume $(iii)$ is satisfied. Then, it follows from Proposition \ref{moduliduality} that there exists $\gamma \in (0,1]$ so that for any $t_0\in (0,1]$ there exists a norm $|\ |$ on $X$ satisfying 
		$$\forall x\in X,\ M^{-1}\|x\|_X \le |x| \le M\|x\|_X\ \ \text{and}\ \ \forall t\in [t_0,1],\ \overline{\theta}^*_{|\ |}(t)\ge \gamma t^q.$$ 
		Fix now $\eps_1,\ldots,\eps_n \in (0,1]$ and pick an equivalent norm $|\ |$ as above for $t_0=\min\{\frac{\eps_1}{4M^2},\ldots,\frac{\eps_n}{4M^2}\}$. Assume that $s_{\eps_1}\ldots s_{\eps_n}B_{X^*}$ is not empty. Then $s_{\eps_1}\ldots s_{\eps_n}(MB_{|\ |^*})$ is non empty and by homogeneity, so is $s_{\frac{\eps_1}{M}}\ldots s_{\frac{\eps_n}{M}}(B_{|\ |^*})$. Thus, if we  denote $\sigma_\eps$ the Szlenk derivation on $X^*$ where the diameter is taken with respect to the norm $|\ |$, we have that  $\sigma_{\frac{\eps_1}{M^2}}\ldots \sigma_{\frac{\eps_n}{M^2}}B_{|\ |^*}$ is non empty. Then classical manipulations on the Szlenk derivation imply that 
		$$\frac12B_{|\ |^*} \subset \sigma_{\frac{\eps_1}{4M^2}}\ldots \sigma_{\frac{\eps_n}{4M^2}}B_{|\ |^*} \subset \prod_{k=1}^n\big(1-\frac{\gamma\eps_k^q}{4^qM^{2q}}\big) B_{|\ |^*}.$$
		The argument for the first inclusion can be found in \cite{Lancien2006} (proof of Proposition 3.3) and the second inclusion follows from the definition of $\overline{\theta}^*_{|\ |}$ and homogeneity. Finally we use the fact that $t\le -\log(1-t)$, for $t\in [0,1)$ and elementary calculus to deduce that $\sum_{k=1}^n \eps_k^q \le \frac{4^qM^{2q}}{\gamma} \log 2$. This finishes the proof. 
		
		\smallskip
		We now turn to $(iv)$ implies $(ii)$. This was already proved in \cite{CauseyIllinois2018} in a more general setting. We include the simpler proof in our situation for the sake of clarity. So, let $M$ be such that if $\ee_1, \ldots, \ee_n\geqslant 0$ are such that $s_{\ee_1}\ldots s_{\ee_n}(B_{X^*})\neq 0$, then $\sum_{i=1}^n \ee_i^q \leqslant M^q$. Let $D$ be a weak neighborhood base of $0$ in $X$ and assume that $c>0$ is such that, for some $n\in\nn$ and $(x_t)_{t\in D^{\leqslant n}}$ weakly null in  $B_X$ we have that for each $t\in D^n$, there exists $(a_i)_{i=1}^n\in B_{\ell_p^n}$ satisfying $\|\sum_{i=1}^n a_ix_{t|_i}\|>c$.   For each $t\in D^n$, fix $x^*_t\in B_{X^*}$ and $(a^t_i)_{i=1}^n\in B_{\ell_p^n}$ such that \[\text{Re\ } x^*_t\Bigl(\sum_{i=1}^n a_i^tx_{t|_i}\Bigr) = \Bigl\|\sum_{i=1}^n a_i^tx_{t|_i} \Bigr\| > c.\]  Define $f:D^n\to  B_{\ell_\infty^n}$ by letting $f(t) = ( x^*_t(x_{t|_1}), \ldots, x^*_t(x_{t|_n}))$. Fix $\delta>0$ arbitrary.  By Corollary \ref{prune2}, there exist  $(b_i)_{i=1}^n\in B_{\ell_\infty^n}$ and $\theta:D^n\to D^n$ preserving lengths and initial segments such that for all $t\in D^n$ and $1\leqslant i\leqslant n$,  
		\begin{enumerate}[(i)]
			\item if $\theta((U_1, \ldots, U_k))=(V_1, \ldots, V_k)$, then $V_k\subset U_k$, and 
			\item $|x^*_{\theta(t)}(x_{\theta(t)|_i})-b_i|<\delta$. 
		\end{enumerate} 
		By replacing $x_s$ with $x_{\theta(s)}$ and $x^*_t$ with $x^*_{\theta(t)}$ for each $s\in D^{\leqslant n}$ and $t\in D^n$, we can relabel and assume that the original collections $(x_t)_{t\in D^{\leqslant n}}$ and $(x^*_t)_{t\in D^n}\subset B_{X^*}$ satisfy this property. 
		
		Define $\ee_i = \max\{0,|b_i|-2\delta\}$ for each $1\leqslant i\leqslant n$. We will prove that for each $0\leqslant j\leqslant n$ and $t\in D^{n-j}$, there exists $x^*\in s_{\ee_{n-j+1}}\dots s_{\ee_n}(B_{X^*})$, and if $j<n$, this $x^*$ can be chosen such that  for each $1\leqslant i\leqslant n-j$, $|x^*(x_{t|_i})-b_i|\leqslant \delta$. We prove this claim by induction on $j$. By convention, in the $j=0$ case, $s_{\ee_{n+1}}s_{\ee_n}(B_{X^*})=B_{X^*}$ and we just take $x^*=x^*_t\in B_{X^*}$.  Next, assume the result holds for some $0\leqslant j< n$.   By the inductive hypothesis, for each $t\in D^{n-j-1}$ and $U\in D$, since $t\smallfrown(U)\in D^{n-j}$,  there exists $x^*_U\in s_{\ee_{n-j+1}}\ldots s_{\ee_n}(B_{X^*})$ such that for each $1\leqslant i\leqslant n-j-1$, $|x^*_U(x_{t|_i})-b_i|\leqslant \delta$ and $|x^*_U(x|_{t\smallfrown (U)})-b_{n-j}| \leqslant \delta$.   If $\ee_{n-j}=0$, we pick any $U$ in $D$ and set $x^*=x^*_U$.  Note that the conclusions are satisfied, since, by convention \[x^*_U\in s_{\ee_{n-j+1}}\ldots s_{\ee_n}(B_{X^*})=s_{\ee_{n-j}}s_{\ee_{n-j+1}}\ldots s_{\ee_n}(B_{X^*}).\] 
		Consider now the case $\ee_{n-j}>0$.   If $x^*$ is any weak$^*$-cluster point of $(x^*_U)_{U\in D}$, then clearly $|x^*(x_{t|_i})-b_i| \leqslant \delta$ for each $1\leqslant i\leqslant n-j-1$. Note also that, since $(x_{t\smallfrown (U)})_{U\in D}$ is weakly null, there exists $U_0\in D$ such that for all $U\subset U_0$, $|x^*(x_{t\smallfrown (U)})|<\delta$. This implies that  
		\[\forall U\subset U_0,\ \ \|x^*_U-x^*\|\ge |(x^*_U-x^*)(x_{t\smallfrown (U)})| > |b_{n-j}|-2\delta =\ee_{n-j}.\]
		We now use that $x^*$ is a weak$^*$-cluster point of $(x^*_U)_{U\subset U_0}\subset s_{\ee_{n-j+1}}\ldots s_{\ee_n}(B_{X^*})$ to deduce that $x^*\in s_{\ee_{n-j}}\ldots s_{\ee_n}(B_{X^*})$. This finishes the inductive proof of our claim.  Applying this claim for $j=n$ yields the existence of some $x^*\in s_{\ee_1}\ldots s_{\ee_n}(B_{X^*})$, from which it follows that $\sum_{i=1}^n \ee_i^q\leqslant M^q$. We can now use this information to estimate the constant $c$. We define $I=\{i\leqslant n: |b_i|>2\delta\}$.  Then, for any $t\in D^n$, 
		\begin{align*}  c &< \text{Re\ }x^*_t\Bigl(\sum_{i=1}^n a_i^t x_{t|_i}\Bigr)  \leqslant \delta n +  \sum_{i=1}^n |a_i^t| |b_i| \leqslant 3\delta n + \sum_{i\in I}|a_i^t||b_i|\\ 
			&\leqslant 5\delta n +\sum_{j\in I}|a_i^t|\ee_i  \leqslant 5\delta n + \|(a_i^t)_{i\in I}\|_{\ell_p^n}\|(\ee_i)_{i\in I}\|_{\ell_q^n} \leqslant 5\delta n + M.\end{align*} Since $\delta>0$ was arbitrary, we conclude that $c\leqslant M$. This finishes the proof of this last implication. 
		
	\end{proof}

	Next we describe the class $\textsf{{N}}_p$. A more general version of the following result in proved in \cite{Causey3.5}. 
	
	\begin{theorem}\label{ntheorem} 
		Fix $1<p\leqslant \infty$ and let $q$ be conjugate to $p$. Let $X$ be a Banach space. The following are equivalent. 
		\begin{enumerate}[(i)]
			\item $X\in \textsf{\emph{N}}_p$. 
			\item There  exists a constant $K>0$ such that for any $n\in\nn$ and any weakly null collection  $(x_t)_{t\in D^{\le n}}$ in $B_X$, there exists $t\in D^n$ such that $\|\sum_{i=1}^nx_{t|_i}\|\le Kn^{1/p}$.
			\item There exists a constant $M\ge 1$ and a constant $c>0$ such that for each $\sigma\in (0,1]$, there exists a  norm $|\ |$ on $X$ such that $M^{-1}|x|\le \|x\|_X\le M|x|$ for all $x\in X$ and
			\begin{enumerate}[(a)]
				\item if $1<p<\infty$, $\overline{\rho}_{|\ |}(\sigma)\leqslant c \sigma^p$ 
				\item if $p=\infty$, $\overline{\rho}_{|\ |}(c) \leqslant \sigma$. 
			\end{enumerate}  
		\end{enumerate}
	\end{theorem}
	
	\begin{proof} Again, the equivalence between $(i)$ and $(ii)$ follows from our initial discussion on games.
		
		\smallskip
		The argument for  $(ii) \Rightarrow (iii)$ is an adaptation of the proof of Theorem 4.2 in \cite{GKL2001} to the non-separable case. We refer the reader to this paper or to  \cite{Causey3.5}.
		
		\smallskip Let us briefly explain the simple implication $(iii) \Rightarrow (ii)$. Let us assume, as we may, that $\|\ \|$ satisfies $(iii)$ for $\sigma=\frac12$. We shall show the existence of a constant $C\ge 2$ such that $(ii)$ is satisfied. Let $(x_t)_{t\in D^{\le n}}$ be a weakly null tree in $B_X$. Pick $t\in D^{\le n}$ such that $2<\|\sum_{i}^k x_{t|_i}\|\le 3$ (if this is not possible we are done). Now we pick recursively $U_{k+1},\ldots,U_{n}$ so that for all $k<l\le n$, we have, if we denote $s=t\smallfrown (U_{k+1},\cdots U_{n})$,  $\|\sum_{i}^l x_{s|_i}\|>2$ and $\|\sum_{i}^l x_{s|_i}\|\le \|\sum_{i}^{l-1} x_{s|_i}\|(1+2c2^{-p})$. It now follows from classical use of Orlicz functions (see for instance the proof of Theorem 6.1 in \cite{KaltonRandrianarivony2008}) that there exist a constant $K>0$ so that  $\|\sum_{i=1}^nx_{t|_i}\|\le Kn^{1/p}$.
	\end{proof}

	We recall that $\textsf{{P}}_p$ is defined to be $\bigcap_{1<r<p}\textsf{{T}}_r$. Then we have. 
	\begin{theorem}\label{ptheorem} 
		Fix $1<p\leqslant \infty$ and let $q$ be conjugate to $p$. Let $X$ be a Banach space. The following are equivalent  
		\begin{enumerate}[(i)]
			\item $X\in \textsf{\emph{P}}_p$. 
			\item For each $1<r<p$, $X$ is $r$-AUS-able. 
			\item There exists an equivalent norm $|\cdot|$ on $X$ such that for all $1<r<p$, $X$ is $r$-AUS. 
			\item For each $1<r<p$, $\theta_n(X)=o(n^{1/r})$.  
			\item $\textsf{\emph{p}}(X)\leqslant q$. 
		\end{enumerate}
	\end{theorem}

	\begin{proof} The equivalence between $(i)$ and $(ii)$ follows from Theorem \ref{ttheorem}. The fact that $(ii)$ implies $(v)$ follows from Proposition \ref{moduliduality}. The implication $(v) \Rightarrow (iii)$ is proved in \cite{CauseyTAMS2019} in a very general setting (non-separable, for higher ordinals and operators). Obviously $(iii)$ implies $(ii)$. The implication $(ii) \Rightarrow (iv)$ also follows from Theorem \ref{ttheorem}. Finally $(iv) \Rightarrow (ii)$ relies on an averaging of the norms provided by $(iii)$ in  Theorem \ref{ntheorem}. 
	\end{proof}
	
	We finally summarize what is known about the inclusions between these classes.
	\begin{theorem}\label{containments} 
		Recall that  $\textsf{{D}}_1$  denotes the class the of all Banach spaces with Szlenk index at most $\omega$. Then
		\begin{enumerate}[(i)]
			\item $\textsf{\emph{D}}_1=\bigcup_{1<p\leqslant \infty}\textsf{\emph{T}}_p=\bigcup_{1<p\leqslant \infty}\textsf{\emph{A}}_p=\bigcup_{1<p\leqslant \infty}\textsf{\emph{N}}_p=\bigcup_{1<p\leqslant \infty}\textsf{\emph{P}}_p$. 
			\item For $1<p<\infty$, $\textsf{\emph{T}}_p\subsetneq \textsf{\emph{A}}_p\subsetneq \textsf{\emph{N}}_p\subsetneq \textsf{\emph{P}}_p$. 
			\item $\textsf{\emph{T}}_\infty\subsetneq \textsf{\emph{A}}_\infty= \textsf{\emph{N}}_\infty\subsetneq \textsf{\emph{P}}_\infty$.
		\end{enumerate}
	\end{theorem}
	
	\begin{proof} Let $1<p\le \infty$. We clearly have that $\textsf{{T}}_p\subset \textsf{{A}}_p\subset \textsf{{N}}_p\subset \textsf{{P}}_p$. It follows from $(iv)$ in Theorem \ref{ttheorem} and $(ii)$ in Theorem \ref{ptheorem} that $\textsf{{P}}_p \subset \textsf{{D}}_1$. We have already explained that if $X\in \textsf{{D}}_1$, then $\textsf{{p}}(X)<\infty$, so Theorem \ref{ptheorem} implies that $X\in \textsf{{T}}_r$, for some $1<r<p$. Our statement $(i)$ follows from gathering all these pieces of information.
		
		The fact that the inclusions are strict in $(ii)$, as well as $\textsf{{T}}_\infty \neq \textsf{{A}}_\infty = \textsf{{N}}_\infty \neq \textsf{{P}}_\infty$ are proved in \cite{Causey3.5}.

	\end{proof}
	
	\section{Separable determination}
	
	We start with a simple but fundamental statement about selecting  weakly null sequences from weakly null nets in AUS-able Banach spaces.
	
	\begin{proposition}\label{baba} 
		Let $X$ be a Banach space with $Sz(X)\leqslant \omega$. Let $D$ be a weak neighborhood base at $0$ in $X$. For any $(x_U)_{U\in D}\subset B_X$ such that $x_U\in U$ for all $U\in D$, there exists a function $f:\nn\to D$ such that $(x_{f(n)})_{n=1}^\infty$ is a weakly null sequence. 
	\end{proposition}
	
	\begin{proof} Since $Sz(X)\leqslant \omega$, $X\in \textsf{T}_r$ for some $1<r<\infty$. Let $1/r+1/s=1$ and $c>\textsf{t}_r(X)$.    Let $\phi$ be a winning strategy in the $T(c,r)$ game.  Let $V_1$ be determined by $\phi$ and fix $U_1\in D$ such that $U_1\subset V_1$.  Let Player II choose $x_{U_1} \in U_1 \cap B_X$. Let $V_2$ be determined by $\phi$ and fix $U_2\in D$ such that $U_2\subset V_2$. Let Player II choose $x_{U_2} \in U_2 \cap B_X$.  Continue in this way until $U_1, U_2, \ldots$ have been chosen.   Define $f(n)=U_n$ and note that $\|(x_{f(n)})_{n=1}^\infty\|_s^w =\|(x_{U_n})_{n=1}^\infty\|_s^w \leqslant c< \infty$.  Therefore $(x_{f(n)})_{n=1}^\infty$ is weakly null.   
	\end{proof}
	
	We are now ready to give a unified proof of the separable determination of all the properties considered in this paper. Before to state it, let us mention that summable Szlenk index and having power type Szlenk index was proved to be separably determined by Draga and Kochanek in \cite{DragaKochanek}.
	
	\begin{theorem}\label{separabledetermination} 
		If $X$ is a Banach space with $Sz(X)\leqslant \omega$, then for each $1<p\leqslant \infty$, \[\textsf{\emph{t}}_p(X)=\sup \{\textsf{\emph{t}}_p(E): E\leqslant X\text{\ is separable}\},\] and this supremum is attained, although possibly infinite.    The same is true of $\textsf{\emph{a}}_p(X)$,  $\textsf{\emph{n}}_p(X)$, and $\theta_n(X)$. In particular, if $X$ is a Banach space all of whose separable subspaces lie in $\textsf{\emph{T}}_p$, then $X$ lies in $\textsf{\emph{T}}_p$.  The same conclusion holds for $\textsf{\emph{A}}_p$, $\textsf{\emph{N}}_p$, $\textsf{\emph{P}}_p$ and $\textsf{\emph{D}}_1$ 
	\end{theorem}
	
	\begin{proof} It is clear that $\textsf{
			{t}}_p(X)\geqslant\sup \{\textsf{t}_p(E): E\leqslant X\text{\ is separable}\}$. If $c>\textsf{t}(X)$, then there exists a collection $(x_t)_{t\in D^{<\omega}}$ such that for each $\tau\in D^\omega$, $\|(x_{\tau|_i})_{i=1}^\infty\|_q^w>c$.    
		
		First, we build $\varphi:\nn^{<\omega}\to D^{<\omega}$ which preserves lengths and immediate predecessors such that $(x_{\varphi(t)})_{t\in \nn^{<\omega}}$ is weakly null.     We define $\varphi(t)$ by induction on $|t|$.   By Proposition \ref{baba} applied to $(x_{(U)})_{U\in D}$, there exists $f:\nn\to D$ such that $(x_{(f(n))})_{n=1}^\infty$ is weakly null.  Define $\varphi((n))=(f(n))$. Next, if $\varphi(t)$ has been defined, apply Proposition \ref{baba} to $(x_{\varphi(t)\smallfrown (U)})_{U\in D}$ to select $g:\nn\to D$ such that $(x_{\varphi(t)\smallfrown (g(n))})_{n=1}^\infty$ is weakly null.    Define $\varphi(t\smallfrown(n))=\varphi(t)\smallfrown (g(n))$.    This completes the construction. 
		
		Define $y_t=x_{\varphi(t)}$.   It follows that for any $\tau_1\in \nn^\omega$, there exists a unique $\tau\in D^\omega$ such that $\varphi(\tau_1|_i)=\tau|_i$ for all $i\in\nn$, so that \[\|(y_{\tau_1|_i})_{i=1}^\infty\|_q^w = \|(x_{\tau|_i})_{i=1}^\infty\|_q^w > c.\]  Therefore if $F$ is the closed linear span of $(y_t)_{t\in \nn^{<\omega}}$, then $\textsf{t}_p(F)>c$.  This shows that $\textsf{\emph{t}}_p(X)\leqslant\sup \{\textsf{\emph{t}}_p(E): E\leqslant X\text{\ is separable}\}$. Next, let $R$ denote the set of rational numbers $r$ such that $\textsf{t}_p(X)>r$. For each $r\in R$, let $F_r$ be a separable subspace of $X$ such that $\textsf{t}_p(F_r)>r$, and let $E$ be the closed span of $E_r$, $r\in R$. Then $\textsf{t}_p(E)=\textsf{t}_p(X)$, and the supremum is attained. The arguments for $\textsf{a}_p(X), \textsf{n}_p(X), \theta_n(X)$ are similar.


		If $X$ is a Banach space all of whose separable subspaces lie in $\textsf{T}_p\subset \textsf{D}_1$, then $\textsf{t}_p(X)=\sup\{\textsf{t}_p(E):E\leqslant X \text{\ is separable}\}$ must be finite. Indeed, if the supremum were infinite, then since it is attained, there would exist some separable $E\leqslant  X$ such that $\textsf{t}_p(E)=\infty$, and $E$ does not belong to $\textsf{T}_p$.   Similar arguments hold for $\textsf{A}_p$ and $\textsf{N}_p$.  
		
		For $\textsf{P}_p$, we note that \begin{align*} X\in \textsf{P}_p & \Leftrightarrow (\forall 1<r<p)(X\in \textsf{T}_r)  \\ & \Leftrightarrow (\forall 1<r<p)(\forall E\leqslant X\text{\ separable})(E\in \textsf{T}_r) \\ & \Leftrightarrow (\forall E\leqslant X\text{\ separable})(\forall 1<r<p)(E\in \textsf{T}_r)  \\ & \Leftrightarrow (\forall E\leqslant X\text{\ separable})(E\in\textsf{P}_p) \end{align*}  
		
		Assume now that $X$ is not in $\textsf{D}_1$. Then, for any  $p\in \Qdb \cap (1,\infty)$, $X$ does not belong to $\textsf{T}_p$. So for any $p\in \Qdb \cap (1,\infty)$, there exists a separable subspace $E_p$ of $X$ so that  $E_p$ is not in $\textsf{T}_p$. Then the closed linear span of these $E_p$'s is a separable subspace of $X$ which does not belong to $\textsf{D}_1$.
		
	\end{proof}

	\section{Three-space properties}
	
	\subsection{Introduction}
	
	We recall that a property $(P)$ of Banach spaces is a \emph{Three-Space Property} (3SP in short) if it passes to quotients and subspaces and a Banach space $X$ has $(P)$ whenever it admits a subspace $Y$ such that $Y$ and $X/Y$ have $(P)$.

	Note first that the properties considered in this paper pass easily to subspaces and quotients. 
	
	\begin{proposition}\label{stability} Fix $1<p\leqslant \infty$ and $X\in \textsf{\emph{Ban}}$. If $X$ is in any of the classes $\textsf{\emph{T}}_p, \textsf{\emph{A}}_p, \textsf{\emph{N}}_p$, or $\textsf{\emph{P}}_p$, then any subspace, quotient, or isomorph of $X$ lies in the same class. 
	\end{proposition}
	
	\begin{proof} For subspaces and isomorphs, the result is clear. For quotients, the result follows easily from the dual characterizations of these properties, which clearly pass to weak$^*$-closed subspaces of $X^*$ (we recall that $(X/Y)^*$ is canonically isometric to $Y^\perp \subset X^*$ by a weak$^*$-weak$^*$-bi-continuous map).  
	\end{proof}
	
	The following lemma will allow us, when convenient, to reduce our questions to the separable setting. 
	
	\begin{lemma}\label{separable reduction}
		Let $\textsf{\emph{I}}$ be a class of Banach spaces which contains all subspaces, quotients, and isomorphs of its members. Suppose also that membership in $\textsf{\emph{I}}\cap \textsf{\emph{Sep}}$ is a $3$SP, and that if $X$ is a Banach space such that every separable subspace of $X$ lies in $\textsf{\emph{I}}$, then $X$ lies in $\textsf{\emph{I}}$.  Then membership in $\textsf{\emph{I}}$ is a $3$SP. 
	\end{lemma}
	
	\begin{proof} Let $X$ be a Banach space and suppose that $Y$ is a subspace of $X$ such that $Y, X/Y\in \textsf{I}$.   If $X$ is not in $\textsf{I}$, then there exists a separable subspace $E$ of $X$ such that $E$ is not in $\textsf{I}$.   Fix a countable, dense subset $S$ of $E$ and for each $x\in S$, fix a countable subset $R_x$ of $Y$ such that \[\|x\|_{X/Y}= \inf_{y\in R_x}\|x-y\|.\]   Let $G$ denote the closed linear span of \[E\cup \bigcup_{x\in S}R_x\] and let $F$ denote the closed linear span of $\bigcup_{x\in S}R_x$.    Then $F,G$ are separable and $F$, being a subspace of $Y$, lies in $\textsf{I}\cap \textsf{Sep}$.   Moreover, it follows from the construction of $G$ that $G/F$ is isometric to a subspace of $X/Y$, which means $G/F$ also lies in $\textsf{I}\cap \textsf{Sep}$. Therefore $G$ lies in $\textsf{I}$, as does $E\leqslant G$.   Therefore every separable subspace of $X$ lies in $\textsf{I}$, as does $X$. 
	\end{proof}
	
	\subsection{Past results}
	
	It was shown by Draga, Kochanek, and the first-named author in \cite{DKC} that membership in $\textsf{P}_p$ is a $3$SP, although it was not stated in this way. We isolate here a shorter and more direct argument.   We will show the following. 
	
	\begin{theorem}\label{cdk} Fix a Banach space $X$, a closed subspace $Y$ of $X$, and $1<p\leqslant \infty$.  
		\begin{enumerate}[(i)]
			\item If $ \textsf{\emph{n}}_p(Y)$ and $ \textsf{\emph{n}}_p(X/Y)$ are finite, then there exist constants $C, \lambda$ such that for all $2\leqslant n\in\nn$, $\textsf{\emph{n}}_p(X) \leqslant C (\log n)^\lambda$.  
			\item If $Y$ and $X/Y$ have $\textsf{\emph{P}}_p$, so does $X$. \end{enumerate}
	\end{theorem}
	
	The fact that $\textsf{P}_p$ is a 3SP was shown in Theorem $7.5$ of \cite{DKC}. The proof there established an inequality similar to Theorem \ref{cdk}$(i)$, but using  $\textsf{a}_p$ rather than $\textsf{n}_p$. In fact, the argument there was given for asymptotic Rademacher type $p$, which deals with Rademacher averages of arbitrary linear combinations of the branches of weakly null trees, which added significant technicality to the proof.  Because $\textsf{n}_p$ deals  only with flat linear combinations, we sketch the simpler proof below.   
	
	\medskip We will use the following, which is an analogue of a lemma of Enflo, Lindenstrauss, and Pisier in their solution of the Palais problem \cite{ELP}. 
	
	\begin{lemma}\label{cdk2}
		For any Banach space $X$, any closed subspace $Y$ of $X$, and any $m,n\in\nn$, \[\theta_{mn}(X)\leqslant 6\big(\theta_m(X/Y)\theta_n(X)+\theta_m(X)\theta_n(Y)\big).\] 
	\end{lemma}
	
	Let us first deduce Theorem \ref{cdk} from Lemma \ref{cdk2}. 
	
	\begin{proof}[Proof of Theorem \ref{cdk}] 
		$(i)$ By Lemma \ref{cdk2}, for any $n\in\nn$, \begin{align*}\textsf{n}_{p,n^2}(X)& =\frac{\theta_{n^2}(X)}{n^{2/p}} \leqslant 6\Bigl(\frac{\theta_n(X/Y)}{n^{1/p}}+\frac{\theta_n(Y)}{n^{1/p}}\Bigr)\frac{\theta_n(X)}{n^{1/p}} \\ & \leqslant c \textsf{n}_{p,n}(X),\end{align*} where $c=10(\textsf{n}_p(X/Y)+\textsf{n}_p(Y))$.    We argue as in  Theorem $3$ of  \cite{ELP} to deduce the existence of the constants $C$ and $\lambda$. 
		
		$(ii)$ Assume $Y,X/Y$ have $\textsf{P}_p$. Fix $1<r<s<p$.    Since $Y,X/Y$ have $\textsf{P}_p$, they also have $\textsf{N}_s$, which means there exist constants $C,\lambda$ such that for all $2\leqslant n$, $\theta_n(X)=\textsf{n}_{s,n}(X)n^{1/s} \leqslant C (\log n)^\lambda n^{1/s}$.  Then $\textsf{n}_{r,n}(X) =\theta_n(X)n^{-1/r} \leqslant C (\log n)^\lambda n^{1/s-1/r}$, which vanishes as $n$ tends to infinity. Therefore, for $1<r<p$, $\textsf{n}_r(X)<\infty$, so $X\in \textsf{P}_p=\bigcap_{1<r<p}\textsf{N}_r$. 
		
	\end{proof}
	
	We next recall an easy technical piece which we will need for the proof of Lemma \ref{cdk2}. 
	
	\begin{claim} Let $X$ be a Banach space and $Y$ a closed subspace. For any weak neighborhood $U_1$ of $0$ in $X$ and $R, \delta>0$, there exists a weak neighborhood $U_2$ of $0$ in $X$ such that if $x\in U_2\cap RB_X$ with $\|x\|_{X/Y}<\delta$, then there exists $y\in U_1\cap  RB_Y $ such that $\|x-y\|<6\delta$. 
		\label{triv_claim}
	\end{claim}
	
	\begin{proof} If it were not so, then for some weak neighborhood $U_1$ of $0$ in $X$ and some $R, \delta>0$, there would exist a weakly null net $(x_\lambda)\subset RB_X$ such that, for all $\lambda$, $\|x_\lambda\|_{X/Y}<\delta$  and for all $y\in U_1\cap RB_Y$, $\|x_\lambda-y\|\geqslant 6\delta$.  For each $\lambda$, we can fix $y_\lambda\in Y$ such that $\|x_\lambda-y_\lambda\|<\delta$.   By passing to a subnet and relabeling, we can assume $(y_\lambda)$ is weak$^*$-convergent to some $y^{**}\in B_{X^{**}}$.  Fix $\ee>0$ and a finite subset $F$ of $X^*$ such that $V:=\{x\in X:(\forall x^*\in F)(|x^*(x)|<2\ee)\}\subset U_1$.    Since $(x_\lambda)$ is weakly null and $(y_\lambda)$ is weak$^*$-convergent to $y^{**}$, we can find $\lambda_1$,  a finite subset $G$ of the index set of $(x_\lambda)$, and positive numbers $(w_\lambda)_{\lambda\in G}$ summing to $1$ such that \begin{enumerate}[(i)]\item for all $x^*\in F$ and $\lambda\in \{\lambda_1\}\cup G$, $|y^{**}(x^*)-x^*(y_{\lambda})|<\ee$, \item $\|\sum_{\lambda\in G}w_\lambda x_\lambda\| < \delta$. \end{enumerate} Let $y_1=y_{\lambda_1}-\sum_{\lambda\in G}w_\lambda y_\lambda\in V$ and note that \begin{align*} \|y_1-x_{\lambda_1}\|& \leqslant \|y_{\lambda_1}-x_{\lambda_1}\| + \sum_{\lambda\in G}w_\lambda \|y_\lambda-x_\lambda\| + \|\sum_{\lambda\in G}w_\lambda x_\lambda\| < 3\delta.\end{align*}   Since $\|x_{\lambda_1}\| \leqslant R$, $\|y_1\|\leqslant R+3\delta$. If $\|y_1\|\leqslant R$, let $y=y_1$, and otherwise let $y=\frac{R}{\|y_1\|}y_1$, noting that $\|y-x\|\leqslant \|y-y_1\|+\|y_1-x\|< 6\delta$. By convexity of $V$, $y\in V\subset U_1$, and we reach a contradiction.

	\end{proof}
	
	Let us now sketch the proof of Lemma \ref{cdk2}. 
	
	\begin{proof}[Sketch] If $Y$ is finite dimensional, then $\theta_n(Y)=0$ and $\theta_n(X/Y)=\theta_n(X)$ for all $n\in\nn$. Then the inequality follows, without the factor of $6$, using submultiplicativity of $\theta_n(X)$.  A similar conclusion holds if $X/Y$ is finite dimensional. We can therefore assume $Y,X/Y$ are infinite dimensional, and $\theta_n(Y), \theta_n(X/Y)\geq 1$ for all $n\in\nn$.    
		
		Of course, the idea is to consider a weakly null tree indexed by $D^{mn}$ to consist of inner trees of height $m$, and outer trees of height $n$. For $\psi> \theta_m(X/Y)$, $\psi_1>\theta_m(X)$, $\phi>\theta_n(Y)$, and $\phi_1>\theta_n(X)$, we can fix winning strategies $\chi, \chi_1, \varpi$, and $\varpi_1$ for Player I in each of the games $\Theta(\psi,m)$ on $X/Y$, $\Theta(\psi_1,m)$ on $X$, $\Theta(\phi,n)$ on $Y$, and $\Theta(\phi_1,n)$ on $X$, respectively.   For a weakly null collection $(x_t)_{t\in D^{\leqslant mn}}\subset B_X$, we claim that we can recursively select $t_1\in D^m$, $y_1\in Y$, $t_2\in D^{2m}$ such that $t_1\prec t_2$, $y_2\in Y$, $\ldots$, such that, with $t=t_n\in D^{mn}$, for all $1\leqslant i\leqslant n$,  
		\begin{enumerate}[(i)]
			\item  $\|y_i-\sum_{j=(i-1)m+1}^{im} x_{t|_i}\| \leqslant 6\psi$,
			\item $\|\sum_{j=(i-1)m+1}^{im} x_{t|_i}\| \leqslant \psi_1$, 
			\item $\|\sum_{i=1}^n \frac{y_i}{10\psi_1}\|\leqslant \phi$, 
			\item $\|\sum_{i=1}^n \frac{ y_i - \sum_{j=(i-1)m+1}^{im} x_{t|_j}}{10\psi}\| \leqslant \phi_1$. 
		\end{enumerate}
		Then \begin{align*} \Bigl\|\sum_{i=1}^{mn}x_{t|_i}\Bigr\| & \leqslant  \Bigl\|\sum_{i=1}^m \bigl[y_i-\sum_{j=(i-1)m+1}^{im} x_{t|_j}\bigr]\Bigr\| + \Bigl\|\sum_{i=1}^n y_i\Bigr\|\\ & \leqslant 6 \psi\phi_1 + 6 \psi_1\phi. 
		\end{align*}  
		Since $\psi>\theta_m(X/Y)$, $\psi_1>\theta_m(X)$, $\phi>\theta_n(Y)$, and $\phi_1>\theta_n(X)$ were arbitrary, this will yield the inequality.  
		
		We now explain how to choose $t_i$ and $y_i$. Assume that for some $k<n$,  we have already chosen $t_1 \prec \ldots \prec t_k$, $t_i\in D^{im}$, and $y_1, \ldots, y_k$.  Assume also that $(y_i/10\psi_1)_{i=1}^k$ and $((y_i-\sum_{j=(i-1)m+1}^{im} x_{t_k|_j})/10\psi)_{i=1}^k$ have been chosen by Player II against Player I, who is using strategies $\varpi$ and $\varpi_1$, respectively.    Let $U, U_1$ be the weak neighborhoods chosen for the next stage of the game by strategies $\varpi$ and $\varpi_1$, respectively.  By replacing these sets with subsets if necessary, we can assume they are convex. 
		
		By Claim \ref{triv_claim}, there exists a weak neighborhood $W$ of $0$ in $X$, which we can also assume is convex,  such that if $x\in W\cap \psi_1 B_X$ satisfies $\|x\|_{X/Y}\leqslant \psi$, then there exists $y\in U\cap \frac{1}{2}U_1\cap \psi_1 B_Y$ such that $\|y-x\|\leqslant 6\psi$.   Let $Q:X\to X/Y$ denote the quotient map and, using the strategies $\chi$ and $\chi_1$, choose $t_k\prec s_1\prec \ldots \prec s_m=t_{k+1}\in D^{(k+1)m}$ such that for each $1\leqslant j\leqslant m$, \[x_{s_j} \in G_j\cap Q^{-1}(H_j) \cap \frac{1}{m} W \cap \frac{1}{2m}U_1.\]    Here, the sets $H_j$ are determined by $\chi$ playing against Player II's choices of $x_{s_1}+X/Y, \ldots, x_{s_m}+X/Y$ and the sets $G_j$ are determined by $\chi_1$ playing against Player II's choices of $x_{s_1}$, $\ldots$, $x_{s_m}$.  Note that $G_j\cap Q^{-1}(H_j)$ is a weak neighborhood of $0$ in $X$.  Playing according to $\chi$ and $\chi_1$ guarantees that 
		\[\Big\|\sum_{j=km+1}^{(k+1)m} x_{t_{k+1}|_j}\Big\|=\Big\|\sum_{j=1}^m x_{s_j}\Big\|\leqslant \psi_1\] 
		and 
		\[\Big\|\sum_{j=km+1}^{(k+1)m} x_{t_{k+1}|_j}\Big\|_{X/Y}=\Big\|\sum_{j=1}^m x_{s_j}\Big\|_{X/Y}\leqslant \psi.\]  
		Since $\sum_{j=1}^m x_{s_j}\in \frac{1}{m}V+\ldots + \frac{1}{m}V=V$, there exists $y_{k+1}\in U\cap \frac{1}{2}U_1\cap \psi_1 B_Y$ such that 
		$$\Big\|y_{k+1}-\sum_{j=1}^m x_{s|_j}\Big\| \leqslant 6 \psi.$$  
		Note also that  $y_{k+1}-\sum_{j=1}^m x_{s_j}\in \frac{1}{2}U_1 + \frac{1}{2m}U_1+\ldots + \frac{1}{2m}U_1=U_1$.     Therefore 
		$$\frac{y_{k+1}}{\phi_1} \in U\cap B_Y\ \ \text{and}\ \ \frac{y_{k+1}-\sum_{j=km+1}^{(k+1)m}x_{t_{k+1}|_j}}{10\psi} \in U_1\cap B_X$$ have been chosen in accordance with the winning strategies $\varpi$ and $\varpi_1$, respectively.  This completes the recursive choices. Items $(i)$ and $(ii)$ are seen to be satisfied from the construction, while items $(iii)$ and $(iv)$ follow from the fact that the outer sequences were chosen according to $\chi$ and $\chi_1$.

	\end{proof}

	\subsection{A counterexample}
	
	For $p\in (1,\infty)$, contrary to $\textsf{P}_p$, none of the properties $\textsf{T}_p$, $\textsf{A}_p$, $\textsf{N}_p$ is a three-space property.
	Before proving this result by giving a counterexample, let us introduce the following definition.
	
	\begin{definition}
		Let $X$ be a Banach space and $p \in (1,\infty )$. We say that $X$ has the \textit{weak $p$-Banach-Saks property} if there exists a positive constant $C$ such that for every weakly null sequence $(x_n)$ in $B_X$ and every $k \in \nn$, we can find a subsequence $(x_{n_j})_j$ of $(x_n)$ such that
		\[ \Big\| \sum_{j=1}^k x_{n_j} \Big\| \leq Ck^{1/p} \]
		for all $n_1 < \cdots < n_k$. 
	\end{definition}
	
	Let us notice that every Banach space with property $\textsf{N}_p$, $1 < p < \infty$, has the weak $p$-Banach-Saks property. For instance, use item $(iii)$ of Theorem \ref{ntheorem} and mimic the argument of $(iii)\Rightarrow (ii)$ in the proof of that statement.
	
	\begin{proposition}\label{counterexample} Let $p\in (1,\infty)$. Then the properties $\textsf{\emph{T}}_p$, $\textsf{\emph{A}}_p$, and $\textsf{\emph{N}}_p$ are not three space properties. 
	\end{proposition}
	
	\begin{proof} 
		Let us consider the Kalton-Peck reflexive spaces $Z_p$ (see \cite{KP} or \cite{livre-3SP}), that satisfies the following: $Z_p$ may be normed in such a way that it has a closed subspace $M$ isometric to $\ell_p$ with $Z_p / M$ also isometric to $\ell_p$. It is known that $Z_p$ does not have the weak $p$-Banach-Saks property (see  \cite{livre-3SP}). Hence, since $\ell_p$ has property $\textsf{T}_p$, we get the result by combining the previous remark and item $(ii)$ of Theorem \ref{containments}.
	\end{proof}
	
	In fact, we can even prove that $Z_p$ does not have any of the concentration properties indexed by $p$ considered in \cite{fov} and therefore deduce the following. 
	
	\begin{proposition} \label{propconc}
		Let $p \in (1, \infty)$. Then the properties HFC$_p$, HIC$_p$, HC$_p$, HFC$_{p,d}$, and HC$_{p,d}$ introduced in \cite{fov} are not three spaces properties.
	\end{proposition}
	
	For clarity, we will only define one of them here. First, for $k \in \nn$, let us denote $[\nn]^k = \{ \n=(n_1, \cdots, n_k) \in \nn^k ; n_1 < n_2 < \cdots < n_k \}$. We endow $[\nn]^k$ with the following distance, called the Hamming distance:
	\[ \forall \n, \m \in [\nn]^k, \ d_{\mathbb{H}}(\n,\m)= | \{ 1 \leq j \leq k ; n_j \neq m_j \} | . \] 
	We say that a Banach space $X$ has property HFC$_p$ if there exists $\lambda > 0$ such that, for every $k \in \nn$ and every $1$-Lipschitz function $f : ([\nn]^k, d_{\mathbb{H}}) \to X$, one can find an infinite subset $\mathbb{M}$ of $\nn$ so that
	\[ \forall \overline{n}, \overline{m} \in [\mathbb{M}]^k, \hspace{2mm}  \|f(\overline{n})-f(\overline{m})\| \leq \lambda k^{\frac{1}{p}} . \]
	We refer the reader to \cite{fov} for the definitions of the other concentration properties.
	
	\begin{proof}[Proof of Proposition \ref{propconc}]
		We use the notation from Theorem 6.1 \cite{KP}.
		For $k \in \nn$, the map $g : \left\lbrace \begin{array}{lll}
			[\nn]^k & \to & Z_p \\
			\n & \mapsto & \sum_{j=1}^k u_{n_j}
		\end{array} \right.$ is $2$-Lipschitz and satisfies
		\[ \| g(\n)-g(\m) \| = 2^{1/p} \left( \frac{\ln(2k)}{p}+1 \right) k^{1/p} \]
		for all integers $n_1 < m_1 < n_2 < \cdots < n_k < m_k$. Therefore, $Z_p$ does not have any of the properties mentioned above. The result follows from the fact that $\ell_p$ has them all, which is due to Kalton and Randrianarivony \cite{KaltonRandrianarivony2008}.
	\end{proof}
	
	\subsection{Asymptotic uniform flatenability}
	
	In the case $p=\infty$, the situation is different. First, we easily have the following.
	
	\begin{theorem}\label{AUF3SP}
		The property $\textsf{\emph{T}}_\infty$ is a three-space property.
	\end{theorem}
	
	\begin{proof} Let us first recall that a separable Banach space is $\textsf{T}_\infty$ if and only if it is isomorphic to a subspace of $c_0$ (see \cite{GKL2000}). It was shown by Johnson and Zippin  \cite{JohnsonZippin1974}, who attributed it to Lindenstrauss,  that being isomorphic to a subspace of $c_0$ is a three-space property. So we deduce from Theorem \ref{separabledetermination}, Proposition \ref{stability} and Lemma \ref{separable reduction} that $\textsf{T}_\infty$ is a three-space property.
	\end{proof}
	
	By looking at the argument in \cite{JohnsonZippin1974}, we can actually show slightly more.
	
	\begin{proposition} Let $p\in (1,\infty]$ and $\textsf{\emph{B}}$ be any one of the properties $\textsf{{\emph{T}}}_p$, $\textsf{{\emph{A}}}_p$ and $\textsf{{\emph{N}}}_p$. Let $X$ be a Banach space with a closed subspace $Y$ such that $Y$ has $\textsf{{\emph{T}}}_\infty$ and $X/Y$ has $\textsf{{\emph{B}}}$. Then $X$ has $\textsf{{\emph{B}}}$.
	\end{proposition}
	
	\begin{proof} By Theorem \ref{separabledetermination}, we may assume that $X$ is separable. Let $T:Y \to c_0$ be a linear embedding. It follows from Sobczyk's theorem that $c_0$ has the separable extension property. Therefore $T$ extends to a bounded linear map $S:X \to c_0$. Define now $U:X\to c_0\oplus X/Y$ by $Ux=(Sx,Qx)$ where $Q:X \to X/Y$ is the quotient map. It then easy to check that $U$ is a linear embedding from $X$ into $c_0\oplus X/Y$. Finally, since $\textsf{B}$ passes clearly to direct sums, we deduce that $c_0\oplus X/Y$ and therefore $X$ have $\textsf{B}$.
	\end{proof}
	
	\subsection{Summable Szlenk index}
	This subsection contains the proof of our main result on three-space properties. We will show that  $\textsf{A}_\infty$ is a three-space property.
	
	Recall that a Banach space is said to have property $\textsf{A}_\infty$ provided there exists a constant $c>0$ such that for each $n\in\nn$, Player I has a winning strategy in the $N(c,\infty,n)$ game (since $\textsf{{A}}_\infty= \textsf{{N}}_\infty$ according to Theorem \ref{containments}): Player I chooses a weak neighborhood $U_1$ of $0$ in $X$ and Player II chooses $x_1\in U_1\cap B_X$. Player I chooses a weak neighborhood $U_2$ of $0$ in $X$ and Player II chooses $x_2\in U_2\cap B_X$. Play continues in this way until $x_1, \ldots, x_n$ have been chosen. Player I wins if $\|\sum_{i=1}^n x_i\|\leqslant c$ and Player II wins otherwise.  
	
	As we are going to use a separable reduction, we will be able to use trees indexed by $\Ndb$. For that purpose we let $T_n=\Ndb^{\le n}$.
	
	It will be convenient for us to introduce the notions of \emph{$\uuu$-weakly (or weak$^*$) null} sequences or collections indexed by $T_n$ in a Banach space, where $\uuu$ is a given free ultrafilter on $\nn$. Let us assume, more generally, that $\uuu$ is a filter on $\nn$.  Given a Banach space $X$, we say that a sequence $(x_i)_{i=1}^\infty \subset X$ is $\uuu$-\emph{weakly null} if it converges to $0$ over $\uuu$ in the weak topology.  The notion of $\uuu$-\emph{weak}$^*$-\emph{null} for a sequence $(x^*_i)_{i=1}^\infty\subset X^*$ is defined similarly. A collection $(x_t)_{t\in T_n}\subset X$ is $\uuu$-\emph{weakly null} provided that for each $t\in \{\varnothing\}\cup T_{n-1}$, $(x_{t\smallfrown (m)})_{m=1}^\infty$ is $\uuu$-weakly null. We say $(x^*_t)_{t\in T_n}\subset  X^*$ is $\uuu$-\emph{weak}$^*$-\emph{null} provided that for each $t\in \{\varnothing\}\cup T_{n-1}$, $(x^*_{t\smallfrown (m)})_{m=1}^\infty$ is $\uuu$-weak$^*$-null. Note that for each Banach space $X$ and each $n\in\nn$, $B_X$ admits a $\uuu$-weakly null collection and $B_{X^*}$ admits a $\uuu$-weak$^*$-null collection, namely the collections consisting entirely of zeros. 
	
	In the remainder of this section, we shall always assume that $\uuu$ is a free ultrafilter on $\nn$.
	
	\begin{proposition} \label{normingseq} 
		Let $Z$ be a Banach space such that $Z^*$ is separable, and $(z^*_m)_{m=1}^\infty \subset Z$ a $\uuu$-weak$^*$-null sequence in $Z^*$. For any $\delta>0$, there  exists a $\uuu$-weakly null sequence $(z_m)_{m=1}^\infty \subset B_Z$ such that 
		$$\underset{m\in\uuu}{\lim} \text{\emph{Re}\ }z^*_m(z_m)\geqslant \underset{m\in\uuu}{\lim} \frac{\|z^*_m\|}{2}-\delta.$$
	\end{proposition}
	
	\begin{proof} If $\underset{m\in\uuu}{\lim}\|z^*_m\|=0$, simply take $z_m=0$ for all $m\in\nn$. Suppose that $r:=\underset{m\in\uuu}{\lim} \|z^*_m\|>0$. For each $m\in\nn$, fix $x_m\in B_X$ such that $\text{Re\ }z^*_m(x_m) > \|z^*_m\|-\delta$. Let $x^{**} = \text{weak}^*\text{-}\underset{m\in\uuu}{\lim} x_m$, where the limit is taken in $B_{X^{**}}$. Since $X^*$ is separable, the weak$^*$-topology on $B_{X^{**}}$ is metrizable, which means some subsequence $(u_m)_{m=1}^\infty$ of $(x_m)_{m=1}^\infty$ is weak$^*$-convergent to $x^{**}$. Define $u_0=0$. For each $m\in\nn$, let $C_m=\{i\in\nn_0: |z^*_m(u_i)|\leqslant \delta\}$, where $\nn_0=\{0\}\cup \nn$. Note that $0\in C_m$ for all $m\in\nn$. Define $f:\nn\to \nn_0$ by letting $f(m)=\max C_m$ if $\max C_m < m$, and let $f(m)\in C_m\cap [m,\infty)$ be arbitrary if $\max C_m\geqslant m$. Note that $\underset{m\in\uuu}{\lim} f(m)=l \in \nn_0\cup \{\infty\}$, where $\nn_0\cup \{\infty\}$ is the one-point compactification of $\nn_0$. We claim that $l=\infty$. Indeed, if $l<\infty$, then 
		\[f^{-1}(\{l\})\cap \{m\in\nn:|z^*_m(u_{l+1})|<\delta\} \in \uuu,\] 
		and therefore there exists some $l<m_0\in f^{-1}(\{l\})\cap \{m\in\nn: z^*_m(u_{l+1})\}$. But this means that $l+1\in C_{m_0}$ and $l=\max C_{m_0}$, this is a contradiction. 
		
		Define now $z_m = \frac12(x_m - u_{f(m)})\in B_X$. Since $\underset{m\in\uuu}{\lim} f(m)=\infty$ and weak$^*$-$\lim_{m\to \infty}u_m=x^{**}$, weak$^*$-$\underset{m\in\uuu}{\lim} u_{f(m)} = x^{**}$. Therefore weak-$\underset{m\in\uuu}{\lim} z_m=0$. By our choice of $f(m)$, $|z^*_m(u_{f(m)})| \leqslant \delta$ for all $m\in\nn$, and \begin{align*}  \underset{m\in\uuu}{\lim} \text{Re\ }z^*_m(z_m) & \geqslant \underset{m\in\uuu}{\lim}\text{Re\ }\frac12 z^*_m(x_m) -\underset{m\in\uuu}{\lim} \frac12|z^*_m(u_{f(m)})| \geqslant \frac{r}{2}-\delta.\end{align*}

	\end{proof}
	
	We define $\alpha^\uuu_n(X)$ to be the infimum of $a\geqslant 0$ such that for any $\uuu$-weakly null $(x_t)_{t\in T_n}\subset B_X$, \[\underset{m_1\in\uuu}{\lim} \ldots \underset{m_n\in\uuu}{\lim} \Bigl\|\sum_{i=1}^n x_{(m_1, \ldots, m_i)}\Bigr\| \leqslant a.\]   
	We define $\beta^\uuu_n(X)$ to be the infimum of $b>0$ such that for any $\uuu$-weak$^*$-null $(x^*_t)_{t\in T_n} \subset B_{X^*}$, \[b\underset{m_1\in\uuu}{\lim}\ldots \underset{m_n\in\uuu}{\lim} \Bigl\|\sum_{i=1}^n x^*_{(m_1, \ldots, m_i)}\Bigr\| \geqslant \sum_{i=1}^n \underset{m_1\in\uuu}{\lim}\ldots \underset{m_i\in\uuu}{\lim} \|x^*_{(m_1, \ldots, m_i)}\|.\]
	
	The next proposition details a characterization and a dual formulation of $\textsf{A}_\infty$ for spaces with separable dual. 
	\begin{proposition} \label{dual} 
		Let $X$ be a Banach space such that $X^*$ is separable. 
		\begin{enumerate}[(i)]
			\item  $X$ has ${\emph{\textsf{A}}}_\infty$ if and only if $\sup_n \alpha_n^\uuu(X)<\infty$. 
			\item For each $n\in\nn$, $\alpha_n^\uuu(X)\leqslant 2 \beta^\uuu_n(X)$. 
			\item For each $n\in\nn$, $\beta^\uuu_n(X) \leqslant 2\alpha_n^\uuu(X)$.  
		\end{enumerate}
	\end{proposition}
	
	\begin{proof}$(i)$ Since $X^*$ is separable, there exists a metric $d$ on $B_X$ that is compatible with the weak topology. For each $n\in\nn$, let $U_n=\{x\in X:d(x,0)<1/n\}.$ 
		
		First assume that $X$ does not have property $\textsf{A}_\infty$. Then for each $a>0$, there exists $n\in\nn$ such that Player I fails to have a winning strategy in the $A(a,n)$ game. Since the $A(a,n)$ game is determined, Player II must have a winning strategy in the $A(a,n)$ game. We will choose $(x_t)_{t\in T_n}$ according to this winning strategy. First, let $x_{(m)}\in U_m\cap B_X$ be Player II's response if Player I opens the game with $U_m$. For $1<k\leqslant n$ and $t=(m_1, \ldots, m_k)$, let $x_t\in U_{m_k}\cap B_X$ be Player II's response if Players I and II have chosen $U_{m_1}, x_{(m_1)}, \ldots, x_{(m_1, \ldots, m_{k-1})}, U_{m_k}$. For each $t=(m_i)_{i=1}^n\in T_n$, since $U_{m_1}, x_{(m_1)}, U_{m_2}, x_{(m_1, m_2)}, \ldots, U_{m_n}, x_{(m_1, \ldots, m_n)}$ were chosen according to Player II's winning strategy, $\Bigl\|\sum_{i=1}^n x_{(m_1, \ldots, m_i)}\Bigr\|>a$. Therefore \[\underset{m_1\in\uuu}{\lim}\ldots \underset{m_n\in\uuu}{\lim} \Bigl\|\sum_{i=1}^n x_{(m_1, \ldots, m_i)}\Bigr\| \geqslant a.\]  
		Since $x_{t\smallfrown (m)}\in U_m$ for each $t\in \{\varnothing\}\cup T_{n-1}$ and $m\in\nn$, it follows that $(x_{t\smallfrown (m)})_{m=1}^\infty$ is weakly null, and is therefore a $\uuu$-weakly null sequence, for each $t\in \{\varnothing\}\cup T_{n-1}$. Therefore this collection $(x_t)_{t\in T_n}\subset B_X$ witnesses the fact that $\alpha_n^\uuu(X)\geqslant a$. Since $a>0$ was arbitrary, $\sup_n \alpha^\uuu_n(X)=\infty$. By contraposition, if $\sup_n \alpha^\uuu_n(X) < \infty$, $X$ has property $\textsf{A}_\infty$.

		Next, suppose that $X$ has property $\textsf{A}_\infty$. Fix  $a_0>0$ such that for all $n\in\nn$, Player I has a winning strategy in the $A(a_0,n)$ game. Suppose that for some $n\in\nn$, $\alpha_n^\uuu(X)>a_0$.  Fix $\alpha_n^\uuu(X)>a>a_0$. There exists $(x_t)_{t\in T_n}\subset B_X$ which is $\uuu$-weakly null and such that \[\underset{m_1\in\uuu}{\lim}\ldots \underset{m_n\in\uuu}{\lim} \Bigl\|\sum_{i=1}^n x_{(m_1, \ldots, m_i)}\Bigr\| >a.\]  
		Let $V_1$ be Player I's first choice according to a winning strategy in the $A(a_0,n)$ game. This means we can choose 
		\[m_1 \in \{m\in\nn: x_{(m)}\in V_1\}\cap \{m\in \nn: \|x_{(m)}\|>a\}\in\uuu,\ \ \text{if}\ n=1\] and  
		\begin{align*}&m_1\in \{m\in\nn:x_{(m)}\in V_1\}\\
			&\cap \Bigl\{m\in\nn: \underset{m_2\in\uuu}{\lim}\ldots \underset{m_n\in\uuu}{\lim} \Bigl\|x_{(m)}+\sum_{i=2}^n x_{(m, m_2, \ldots, m_i)}\Bigr\|>a\Bigr\}\in\uuu,\ \ \text{if}\ n>1
		\end{align*}
		Let Player II's choice in the $A(a_0, n)$ game be $x_{(m_1)}$. Next, assume that for some $1\leqslant k<n$, $V_1, \ldots, V_k$ and $m_1, \ldots, m_k$ have been chosen such that \begin{enumerate}[(a)]
			\item $V_1, x_{(m_1)}, \ldots, V_k, x_{(m_1, \ldots, m_k)}$ have been chosen in the $A(a_0, n)$ game with Player I playing according to a winning strategy, \item we have the inequality \[\underset{m_{k+1}\in\uuu}{\lim}\ldots \underset{m_n\in\uuu}{\lim}\Bigl\|\sum_{i=1}^n x_{(m_1, \ldots, m_i)}\Bigr\|>a,\]   
		\end{enumerate} 
		We now describe the recursive step of the construction.\\ 
		Assume first that $k+1<n$. Let $V_{k+1}$ be Player I's next choice according to the winning strategy, let Player II's choice be $x_{(m_1, \ldots, m_{k+1})}$, where 
		\begin{align*} & m_{k+1}  \in \{m\in\nn: x_{(m_1, \ldots, m_k, m)}\in V_{k+1}\}\cap \\ 
			&\Bigl\{m\in\nn: \underset{m_{k+2}\in\uuu}{\lim} \ldots \underset{m_n\in\uuu}{\lim}\Bigl\|x_{(m_1, \ldots, m_k, m)}+\sum_{i\neq k+1;1\le i\le n} x_{(m_1, \ldots, m_i)} \Bigr\|>a\Bigr\}\in \uuu 
		\end{align*}
		Assume, for the last step, that $k+1=n$. Let $V_{k+1}=V_n$ be Player I's next choice according to the winning strategy, let Player II's choice be $x_{(m_1, \ldots, m_{n})}$, where
		\begin{align*} & m_{k+1}=m_n  \in \{m\in\nn: x_{(m_1, \ldots, m_k, m)}\in V_{n}\}\cap \\ 
			&\Bigl\{m\in\nn: \Bigl\|x_{(m_1, \ldots, m_k, m)}+\sum_{i=1}^k x_{(m_1, \ldots, m_i)} \Bigr\|>a\Bigr\}\in \uuu 
		\end{align*}
		This completes the recursive construction, from which we it follows that $\|\sum_{i=1}^n x_{(m_1, \ldots, m_i)}\|>a.$
		However, since $V_1, x_{(m_1)}, \ldots, V_n, x_{(m_1, \ldots, m_n)}$ were chosen with Player I playing according to a winning strategy in the $A(a_0, n)$ game, $\|\sum_{i=1}^n x_{(m_1, \ldots, m_i)}\|\leqslant a_0.$  Since $a_0<a$, this is a contradiction. Hence $\sup_n \alpha^\uuu_n(X)\leqslant a_0<\infty$. Therefore if $X$ has property $\textsf{A}_\infty$, $\sup_n \alpha^\uuu_n(X)<\infty$. 
		
		\medskip
		$(ii)$ Fix $a<\alpha_n^\uuu(X)$. Then there exists a $\uuu$-weakly null collection $(x_t)_{t\in T_n}\subset B_X$ such that \[a<\underset{m_1\in\uuu}{\lim} \ldots \underset{m_n\in\uuu}{\lim} \Bigl\|\sum_{i=1}^n x_{(m_1, \ldots, m_i)}\Bigr\|.\]    
		For each $t=(m_1, \ldots, m_n)\in \Ndb^n$, choose $y^*_t\in \frac{1}{2}B_{X^*}$ such that \[\text{Re\ }y^*_t\Bigl(\sum_{i=1}^n x_{(m_1, \ldots, m_i)}\Bigr) = \frac12\Bigl\|\sum_{i=1}^n x_{(m_1, \ldots, m_i)}\Bigr\|.\]    
		Then \[\frac{a}{2}<\underset{m_1\in\uuu}{\lim} \ldots \underset{m_n\in\uuu}{\lim} \text{Re\ }y^*_{(m_1, \ldots, m_n)}\Bigl(\sum_{i=1}^n x_{(m_1, \ldots, m_i)}\Bigr).\]    
		For $t\in \Ndb^n$, set $z^*_t=y^*_t$ and for each $t\in \{\varnothing\}\cup T_{n-1}$, define \[z^*_t = \underset{m_{|t|+1}\in\uuu}{\lim}\ldots \underset{m_n\in\uuu}{\lim} y^*_{t\smallfrown (m_{|t|+1}, \ldots, m_n)},\] 
		where all limits are taken with respect to the weak$^*$-topology. For $(m_1, \ldots, m_k)\in T_n$, define $x^*_{(m_1, \ldots, m_k)} = z^*_{(m_1, \ldots, m_k)}-z^*_{(m_1, \ldots, m_{k-1})}\in B_{X^*}.$ Note that $(x^*_t)_{t\in T_n}\subset B_{X^*}$ is $\uuu$-weak$^*$-null, which implies  that 
		\begin{align*} 
			\sum_{i=1}^n \underset{m_1\in\uuu}{\lim} \ldots \underset{m_i\in\uuu}{\lim} \|x^*_{(m_1, \ldots, m_i)}\| & \leqslant \beta^\uuu_n(X) \underset{m_1\in\uuu}{\lim}\ldots \underset{m_n\in\uuu}{\lim} \Bigl\|\sum_{i=1}^n x^*_{(m_1, \ldots,m_i)}\Bigr\| \\ & = \beta^\uuu_n(X) \underset{m_1\in\uuu}{\lim}\ldots \underset{m_n\in\uuu}{\lim} \|z^*_{(m_1, \ldots, m_n)}-z^*_\varnothing\|\\
			&\leqslant \beta^\uuu_n(X).
		\end{align*} 
		Note that because $(x_t)_{t\in T_n}$ is $\uuu$-weakly null and $(x^*_t)_{t\in T_n}$ is $\uuu$-weak$^*$-null, it holds that for distinct $i,j\in \{1, \ldots, n\}$, \[\underset{m_1\in\uuu}{\lim}\ldots \underset{m_n\in\uuu}{\lim} x^*_{(m_1, \ldots, m_i)}(x_{(m_1, \ldots, m_j)})=0.\]
		Similarly, for each $1\leqslant i\leqslant n$, $\underset{m_1\in\uuu}{\lim}\ldots \underset{m_n\in\uuu}{\lim} z^*_\varnothing(x_{(m_1, \ldots, m_i)})=0.$
		

		Combining the facts above, we can write
		\begin{align*}
			\frac{a}{2} & < \underset{m_1\in\uuu}{\lim}\ldots \underset{m_n\in\uuu}{\lim} \text{Re\ }z^*_{(m_1, \ldots, m_n)}\Bigl(\sum_{i=1}^n x_{(m_1, \ldots, m_i)}\Bigr) \\
			&= \underset{m_1\in\uuu}{\lim}\ldots \underset{m_n\in\uuu}{\lim} \text{Re\ }(z^*_{(m_1, \ldots, m_n)}-z^*_\varnothing)\Bigl(\sum_{i=1}^n x_{(m_1, \ldots, m_i)}\Bigr) \\ 
			& = \underset{m_1\in\uuu}{\lim}\ldots \underset{m_n\in\uuu}{\lim} \text{Re\ }\Bigl(\sum_{i=1}^n x^*_{(m_1, \ldots, m_i)}\Bigr)\Bigl(\sum_{i=1}^n x_{(m_1, \ldots, m_i)}\Bigr) \\
			&=  \sum_{i=1}^n\underset{m_1\in\uuu}{\lim}\ldots \underset{m_n\in\uuu}{\lim}\text{Re\ }x^*_{(m_1, \ldots, m_i)}(x_{(m_1, \ldots, m_i)})\\ 
			&=   \sum_{i=1}^n\underset{m_1\in\uuu}{\lim}\ldots \underset{m_i\in\uuu}{\lim}\text{Re\ }x^*_{(m_1, \ldots, m_i)}(x_{(m_1, \ldots, m_i)}) \\ & \leqslant   \sum_{i=1}^n\underset{m_1\in\uuu}{\lim}\ldots \underset{m_i\in\uuu}{\lim}\|x^*_{(m_1, \ldots, m_i)}\|\leqslant \beta^\uuu_n(X).
		\end{align*}
		Since $a<\alpha_n^\uuu(X)$ was arbitrary, we are done. 
		
		\medskip
		$(iii)$ Fix $b<\beta^\uuu_n(X)$ and $\delta>0$. Then there exists a collection $(x^*_t)_{t\in T_n}\subset B_{X^*}$ which is $\uuu$-weak$^*$-null and \[b\underset{m_1\in\uuu}{\lim}\ldots \underset{m_n\in\uuu}{\lim} \Bigl\|\sum_{i=1}^n x_{(m_1, \ldots, m_i)}^*\Bigr\| < \sum_{i=1}^n \underset{m_1\in\uuu}{\lim}\ldots \underset{m_i\in\uuu}{\lim} \|x^*_{(m_1, \ldots, m_i)}\|.\]  
		Note that this implies that \[\underset{m_1\in\uuu}{\lim}\ldots \underset{m_n\in\uuu}{\lim} \Bigl\|\sum_{i=1}^n x_{(m_1, \ldots, m_i)}^*\Bigr\|>0.\] 
		We define a collection $(x_t)_{t\in T_n}\subset B_X$ which is $\uuu$-weakly null and such that for each $1\leqslant i\leqslant n$, \[\underset{m_1\in\uuu}{\lim}\ldots \underset{m_i\in\uuu}{\lim} \text{Re\ }x^*_{(m_1, \ldots, m_i)}(x_{(m_1, \ldots, m_i)})\geqslant \frac{1}{2}\underset{m_1\in\uuu}{\lim}\ldots \underset{m_i\in\uuu}{\lim} \|x^*_{(m_1, \ldots, m_i)}\|-\delta.\]   To that end, for each  $t\in \{\varnothing\}\cup T_{n-1}$, choose, as it is allowed by Proposition \ref{normingseq}, $(x_{t\smallfrown (m)})_{m=1}^\infty\subset B_X$ to be a $\uuu$-weakly null sequence such that \[\underset{m\in\uuu}{\lim} \text{Re\ }x^*_{t\smallfrown (m)}(x_{t\smallfrown (m)}) \geqslant \frac{1}{2} \underset{m\in\uuu}{\lim} \|x^*_{t\smallfrown (m)}\| - \delta.\]  
		Note that, since $(x_t)_{t\in T_n}$ is $\uuu$-weakly null and $(x^*_t)_{t\in T_n}$ is $\uuu$-weak$^*$-null, we have again that for all $1\le i\neq j \le n$
		\[\underset{m_1\in\uuu}{\lim} \ldots \underset{m_n\in\uuu}{\lim}x^*_{(m_1, \ldots, m_i)}(x_{(m_1, \ldots, m_j)}=0.\]
		Then, we can write 
		\begin{align*} 
			&\frac{b}{2} \underset{m_1\in\uuu}{\lim} \ldots \underset{m_n\in\uuu}{\lim}\Bigl\|\sum_{i=1}^n x_{(m_1, \ldots, m_i)}^*\Bigr\| - \delta n\\ 
			& < \frac{1}{2}\sum_{i=1}^n \underset{m_1\in\uuu}{\lim}\ldots \underset{m_i\in\uuu}{\lim} \|x^*_{(m_1, \ldots, m_i)}\|-\delta n\\ 
			& \leqslant  \sum_{i=1}^n \underset{m_1\in\uuu}{\lim}\ldots \underset{m_i\in\uuu}{\lim} \text{Re\ }x^*_{(m_1, \ldots, m_i)}(x_{(m_1, \ldots, m_i)}) \\  
			&= \underset{m_1\in\uuu}{\lim}\ldots \underset{m_n\in\uuu}{\lim} \text{Re\ }\Bigl(\sum_{i=1}^n x^*_{(m_1, \ldots, m_i)}\Bigr)\Bigl(\sum_{i=1}^n x_{(m_1, \ldots, m_i)}\Bigr) \\ 
			& \leqslant \underset{m_1\in\uuu}{\lim}\ldots \underset{m_n\in\uuu}{\lim} \Bigl\|\sum_{i=1}^n x^*_{(m_1, \ldots, m_i)}\Bigr\|\Bigl\|\sum_{i=1}^n x_{(m_1, \ldots, m_i)}\Bigr\| \\ 
			& \leqslant \alpha_n^\uuu(X) \underset{m_1\in\uuu}{\lim}\ldots \underset{m_n\in\uuu}{\lim} \Bigl\|\sum_{i=1}^n x^*_{(m_1, \ldots, m_i)}\Bigr\|.   
		\end{align*}  
		Since $\underset{m_1\in\uuu}{\lim}\ldots \underset{m_n\in\uuu}{\lim} \Bigl\|\sum_{i=1}^n x_{(m_1, \ldots, m_i)}^*\Bigr\|>0$, and since $\delta>0$ and  $b<\beta^\uuu_n(X)$ were arbitrary, we are done. 
		
	\end{proof}

	\begin{remark} \upshape In item $(i)$ of the preceding proof, we actually showed that if $X^*$ is separable, then for each $n\in\nn$, $\alpha_n^\uuu(X)$ is the infimum of $a>0$ such that Player I has a winning strategy in the $A(a,n)$ game. 
	\end{remark}
	
	We can now turn to the heart of the proof.
	\begin{lemma} \label{a lemma} 
		For any Banach space $X$ with $X^*$ separable and any subspace $Y$ of $X$, \[\alpha_n^\uuu(X)\leqslant 40 \max\{\alpha_n^\uuu(Y), \alpha_n^\uuu(X/Y)\}^2.\] 
	\end{lemma}
	
	\begin{proof}   
		If $X$ is finite-dimensional, then $\alpha^\uuu_n(X)=\alpha^\uuu_n(Y)=\alpha_n(X/Y)=0$, so assume $X$ is infinite-dimensional. In this case, at least one of $Y$, $X/Y$ must also be infinite-dimensional, which means 
		\[b:=\max\{\beta^\uuu_n(Y), \beta_n^\uuu(X/Y)\} \geqslant 1.\]  
		Fix $(x^*_t)_{t\in T_n}\subset B_{X^*}$ $\uuu$-weak$^*$-null. We will define a bounded, $\uuu$-weak$^*$-null  $(y^*_t)_{t\in T_n}\subset Y^\perp$ such that for each $1\leqslant i\leqslant n$, 
		\[\underset{m_1\in\uuu}{\lim}\ldots \underset{m_i\in\uuu}{\lim} \|x^*_{(m_1, \ldots, m_i)}-y^*_{(m_1, \ldots, m_i)}\| \leqslant 2 \underset{m_1\in\uuu}{\lim}\ldots \underset{m_i\in\uuu}{\lim} \|x^*_{(m_1, \ldots, m_i)}\|_{X^*/Y^\perp}.\]   
		To that end, for each $t=(m_1, \ldots, m_i)\in T_n$, fix $w^*_t \in Y^\perp$ such that $$\|x^*_t-w^*_t\|<\|x^*_t\|_{X^*/Y^\perp}+2^{-m_i}$$ 
		and note that $w_t^*\in 3 B_{Y^\perp}$. For $t\in \{\varnothing\}\cup T_{n-1}$, let $v^*_t=\text{weak}^*\text{-}\underset{m\in\uuu}{\lim} w^*_{t\smallfrown (m)}\in 3B_{Y^\perp}$ and let $y^*_{t\smallfrown (m)}=w^*_{t\smallfrown (m)}-v^*_t$. It is clear that $(y^*_t)_{t\in T_n}\subset Y^\perp$ is bounded and $\uuu$-weak$^*$-null. Note that for any $t\in\{\varnothing\}\cup T_{n-1}$, weak$^*$-$\underset{m\in\uuu}{\lim} (w^*_{t\smallfrown (m)}-x^*_{t\smallfrown (m)})=v^*_t-0=v^*_t$. By weak$^*$-lower semicontinuity of the norm, it follows that \[\|v^*_t\|\leqslant \underset{m\in\uuu}{\lim} \|w^*_{t\smallfrown (m)}-x^*_{t\smallfrown (m)}\| = \underset{m\in\uuu}{\lim} \|x^*_{t\smallfrown (m)}\|_{X^*/Y^\perp}.\]    
		Therefore 
		\[\underset{m\in\uuu}{\lim} \|x^*_{t\smallfrown (m)}-y^*_{t\smallfrown (m)}\| \leqslant \underset{m\in\uuu}{\lim} \|x^*_{t\smallfrown (m)}-w^*_{t\smallfrown (m)}\|+\|v^*_t\|\leqslant 2 \underset{m\in\uuu}{\lim} \|x^*_{t\smallfrown (m)}\|_{X^*/Y^\perp}.\]   
		From this it follows that
		\[\underset{m_1\in\uuu}{\lim}\ldots \underset{m_i\in\uuu}{\lim} \|x^*_{(m_1, \ldots, m_i)}-y^*_{(m_1, \ldots, m_i)}\| \leqslant 2 \underset{m_1\in\uuu}{\lim}\ldots \underset{m_i\in\uuu}{\lim} \|x^*_{(m_1, \ldots, m_i)}\|_{X^*/Y^\perp}.\]  
		Since $(y^*_t)_{t\in T_n}\subset Y^\perp = (X/Y)^*$ is $\uuu$-weak$^*$-null and bounded, by homogeneity, we have that
		\[b\underset{m_1\in\uuu}{\lim}\ldots \underset{m_n\in\uuu}{\lim}\Bigl\|\sum_{i=1}^n y^*_{(m_1, \ldots, m_i)}\Bigr\|  \geqslant \sum_{i=1}^n \underset{m_1\in\uuu}{\lim}\ldots \underset{m_i\in\uuu}{\lim} \|y^*_{(m_1, \ldots, m_i)}\|.\]   
		
		We start with the easy case and suppose first that 
		\[\sum_{i=1}^n \underset{m_1\in\uuu}{\lim}\ldots \underset{m_i\in\uuu}{\lim} \|x^*_{(m_1, \ldots, m_i)}\|_{X^*/Y^\perp} \geqslant \frac{1}{1+4b}\sum_{i=1}^n \underset{m_1\in\uuu}{\lim}\ldots \underset{m_i\in\uuu}{\lim} \|x^*_{(m_1, \ldots, m_i)}\|.\]   
		Since $(x^*_t|_Y)_{t\in T_n}\subset B_{Y^*}=B_{X^*/Y^\perp}$ is $\uuu$-weak$^*$-null,  
		\begin{align*} 
			b\underset{m_1\in\uuu}{\lim}\ldots \underset{m_n\in\uuu}{\lim} \Bigl\|\sum_{i=1}^n x^*_{(m_1, \ldots, m_i)}\Bigr\| & \geqslant  b\underset{m_1\in\uuu}{\lim}\ldots \underset{m_n\in\uuu}{\lim} \Bigl\|\sum_{i=1}^n x^*_{(m_1, \ldots, m_i)}\Bigr\|_{X^*/Y^\perp} \\ & \geqslant \sum_{i=1}^n \underset{m_1\in\uuu}{\lim}\ldots \underset{m_i\in\uuu}{\lim} \|x^*_{(m_1, \ldots, m_i)}\|_{X^*/Y^\perp} \\ & \geqslant \frac{1}{1+4b}\sum_{i=1}^n \underset{m_1\in\uuu}{\lim}\ldots \underset{m_i\in\uuu}{\lim} \|x^*_{(m_1, \ldots, m_i)}\|.  
		\end{align*}  
		
		Next suppose that 
		\[\sum_{i=1}^n \underset{m_1\in\uuu}{\lim}\ldots \underset{m_i\in\uuu}{\lim} \|x^*_{(m_1, \ldots, m_i)}\|_{X^*/Y^\perp} < \frac{1}{1+4b}\sum_{i=1}^n \underset{m_1\in\uuu}{\lim}\ldots \underset{m_i\in\uuu}{\lim} \|x^*_{(m_1, \ldots, m_i)}\|.\] 
		Then since $b\geqslant 1$, 
		\begin{align*}
			& b\underset{m_1\in\uuu}{\lim}\ldots \underset{m_n\in\uuu}{\lim} \Bigl\|\sum_{i=1}^n x^*_{(m_1, \ldots, m_i)}\Bigr\|\geqslant b\underset{m_1\in\uuu}{\lim}\ldots \underset{m_n\in\uuu}{\lim} \Bigl\|\sum_{i=1}^n y^*_{(m_1, \ldots, m_i)}\Bigr\|\\
			&-b\sum_{i=1}^n \underset{m_1\in\uuu}{\lim}\ldots \underset{m_i\in\uuu}{\lim} \|x^*_{(m_1, \ldots, m_i)}-y^*_{(m_1, \ldots, m_i)}\| \\ 
			& \geqslant \sum_{i=1}^n \underset{m_1\in\uuu}{\lim}\ldots \underset{m_i\in\uuu}{\lim} \|y^*_{(m_1, \ldots, m_i)}\| -2b\sum_{i=1}^n \underset{m_1\in\uuu}{\lim}\ldots \underset{m_i\in\uuu}{\lim} \|x^*_{(m_1, \ldots, m_i)}\|_{X^*/Y^\perp} \\ 
			& \geqslant \sum_{i=1}^n \underset{m_1\in\uuu}{\lim}\ldots \underset{m_i\in\uuu}{\lim} \|x^*_{(m_1, \ldots, m_i)}\| - \sum_{i=1}^n \underset{m_1\in\uuu}{\lim}\ldots \underset{m_i\in\uuu}{\lim} \|x^*_{(m_1, \ldots, m_i)}-y^*_{(m_1, \ldots, m_i)}\| \\ 
			& - 2b\sum_{i=1}^n \underset{m_1\in\uuu}{\lim}\ldots \underset{m_i\in\uuu}{\lim} \|x^*_{(m_1, \ldots, m_i)}\|_{X^*/Y^\perp} \\ 
			& \geqslant \sum_{i=1}^n \underset{m_1\in\uuu}{\lim}\ldots \underset{m_i\in\uuu}{\lim} \|x^*_{(m_1, \ldots, m_i)}\| -4b\sum_{i=1}^n \underset{m_1\in\uuu}{\lim}\ldots \underset{m_i\in\uuu}{\lim} \|x^*_{(m_1, \ldots, m_i)}\|_{X^*/Y^\perp}  \\ 
			& \geqslant \sum_{i=1}^n \underset{m_1\in\uuu}{\lim}\ldots \underset{m_i\in\uuu}{\lim} \|x^*_{(m_1, \ldots, m_i)}\| -\frac{4b}{1+4b}\sum_{i=1}^n \underset{m_1\in\uuu}{\lim}\ldots \underset{m_i\in\uuu}{\lim} \|x^*_{(m_1, \ldots, m_i)}\| \\ & = \frac{1}{1+4b}\sum_{i=1}^n \underset{m_1\in\uuu}{\lim}\ldots \underset{m_i\in\uuu}{\lim} \|x^*_{(m_1, \ldots, m_i)}\|. 
		\end{align*} 
		Combining the two previous paragraphs we get 
		\[\beta_n^\uuu(X)\leqslant b(1+4b)\leqslant 5b^2 = 5 \max\{\beta_n^\uuu(Y), \beta_n^\uuu(X/Y)\}^2.\] 
		Combining this inequality with items $(ii)$ and $(iii)$ of Proposition \ref{dual} yields \[\alpha_n^\uuu(X)\leqslant 40\max\{\alpha_n^\uuu(Y), \alpha_n^\uuu(X/Y)\}^2.\]
	\end{proof}
	
	We can now state and prove our result.
	\begin{theorem} The property ${\emph{\textsf{A}}}_\infty$ is a three-space property. 
	\end{theorem}
	
	\begin{proof} Assume first that $Y$ is a closed subspace of a Banach space $X$ such that $Y$ and $X/Y$ are in $\textsf{A}_\infty \cap \textsf{Sep}$. Then $Y$ and $X/Y$ are separable Asplund spaces and $Y^*=X^*/Y^\perp$ and $(X/Y)^*=Y^\perp$ are separable. So $X^*$ is separable and we can apply Lemma \ref{a lemma} and item $(i)$ of Proposition \ref{dual} to deduce that $X$ has $\textsf{A}_\infty$. We have shown that membership in $\textsf{A}_\infty \cap \textsf{Sep}$ is a 3SP. It then follows from Theorem \ref{separabledetermination} and Lemma \ref{separable reduction} that $\textsf{A}_\infty$ is a 3SP.
		
	\end{proof}
	
	\begin{remark}
		Since reflexivity is also a three-space property (cf \cite{KreinSmulian1940}), we can therefore deduce that property HFC$_{\infty}$ of \cite{fov}, which is known to be equivalent to being reflexive and asymptotic-$c_0$ (cf \cite{BLMSJIMJ2021}), is a three-space property. However, the question of whether or not property HC$_{\infty}$ is a three-space property remains open.
	\end{remark}

	\section{Non-linear stabilities}
	
	We first define the non-linear equivalences between Banach spaces that we will discuss in this section. 
	
	\begin{definition} Let $(M,d)$ and $(N,\delta)$ be two metric spaces. A map $f:M \to N$ is a 
		\emph{Lipschitz equivalence (or Lipschitz isomorphism)} from $M$ to $N$ if $f$ is a Lipschitz bijection from $M$ to $N$ with Lipschitz inverse. 
		If there exists a Lipschitz equivalence from $M$ to $N$, we say that $M$ and $N$ are {\it Lipschitz equivalent (or Lipschitz isomorphic)} and we denote $M \buildrel {L}\over {\sim} N$.
	\end{definition}
	
	\begin{definition} Let $(M,d)$ and $(N,\delta)$ be two unbounded metric spaces and $f:M\to N$ be a map. We say that $f$ is  \emph{coarse Lipschitz} if there exist $A,B\ge 0$ such that 
		$$\forall x,y \in M,\  \ \delta(f(x),f(y))\le Ad(x,y)+B.$$
		We say that $f$ is a \emph{coarse Lipschitz equivalence} from $M$ to $N$, if it is coarse Lipschitz and there exists a coarse Lipschitz map $g:N\to M$ and a constant $C\ge 0$ such that
		$$\forall x\in M\ \ d\big((g\circ f)(x),x\big)\le C\ \ {\rm and} \ \ \forall y\in N\ \ \delta\big((f\circ g)(y),y\big)\le C.$$
		If there exists a coarse Lipschitz equivalence from $M$ to $N$, we say that $M$ and $N$ are {\it coarse Lipschitz equivalent} and denote $M \buildrel {CL}\over {\sim} N$.
	\end{definition}
	
	This notion of coarse Lipschitz equivalent metric spaces is the same as the notion of quasi-isometric metric spaces introduced by Gromov in \cite{Gromov1987} (see also the book \cite{GhysDelaHarpe} by E. Ghys and P. de la Harpe).
	
	\medskip We now turn to the notion of net in a metric space.
	
	\begin{definition} Let $0<a\le b$. An \emph{$(a,b)$-net} in the metric space $(M,d)$ is a subset $\cal M$ of $M$ such that for every $z\neq z'$ in $\cal M$,  $d(z,z')\ge a$ and for
		every $x$ in $M$, $d(x,\cal M)< b$.\\
		Then a subset $\cal M$ of $M$ is a \emph{net} in $M$ if it is an $(a,b)$-net for some $0<a\le b$.
	\end{definition}
	
	Let us now give two technical equivalent formulations of the notion of coarse Lipschitz equivalence between Banach spaces. We refer to \cite{DaletLancien} or \cite{GLZ2014} for details.
	
	\begin{proposition}\label{CLE} Let $X$ and $Y$ be two Banach spaces and let $f:X\to Y$ be a  map. The following assertions are equivalent.
		
		\begin{enumerate}[(i)]
			\item The map $f$ is a coarse Lipschitz equivalence.
			\item There exist $A_0>0$ and $K\ge 1$ such that for all $A\ge A_0$ and all maximal $A$-separated subset $\cal M$ of $X$, $\cal N=f(\cal M)$ is a net in $Y$ and
			$$\forall x,x'\in \cal M\ \ \ \frac{1}{K}\|x-x'\|\le \|f(x)-f(x')\|\le K\|x-x'\|.$$
			\item There exist two \underline{continuous} coarse Lipschitz maps $\varphi:X\to Y$ and $\psi:Y\to X$ and a constant $C\ge 0$ such that $\|\varphi(x)-f(x)\|\le C$ for all $x$ in $X$ and
			$$\forall x\in X\ \ \|(\psi\circ \varphi)(x)-x\|\le C\ \ {\rm and} \ \ \forall y\in Y\ \ \|(\varphi\circ \psi)(y)-y\|\le C.$$
		\end{enumerate}
	\end{proposition}
	
	The following results were obtained by Godefroy, Kalton, and the third-named author in  \cite{GKL2000} and \cite{GKL2001}.
	
	\begin{theorem} Let $p\in (1,\infty]$. 
		\begin{enumerate}
			\item The class ${\emph{\textsf{T}}}_p$ is stable under Lipschitz equivalences.
			\item The class ${\emph{\textsf{P}}}_p$ is stable under coarse Lipschitz equivalences.
			\item The class ${\emph{\textsf{A}}}_\infty={\emph{\textsf{N}}}_\infty$ is stable under coarse Lipschitz equivalences.
		\end{enumerate}
	\end{theorem}
	
	In fact, statements (2) and (3) are only proved for uniform homeomorphisms in \cite{GKL2001} in the separable case. The adaptation for coarse Lipschitz equivalences relies on characterization $(iii)$ in Proposition \ref{CLE}, which allows to apply the so-called Gorelik principle (see also \cite{DaletLancien} for details). Then the non-separable case can easily be deduced by a standard separable saturation argument combined with the separable determination of these properties. It is then natural to wonder about the non-linear stability of the classes $\textsf{A}_p$ and $\textsf{N}_p$ for $1<p<\infty$. The results we have detailed in Section \ref{relations} together  with a careful examination of the statements in \cite{GKL2001} or \cite{DaletLancien} will allow us to easily obtain strong new stability results. We start with the following.
	
	\begin{theorem} For any $p \in (1,\infty)$, the class ${\emph{\textsf{A}}}_p$ is stable under coarse Lipschitz equivalences.
	\end{theorem}
	
	\begin{proof} Let $X \in \textsf{A}_p$ and $Y$ a Banach space such that there exists a coarse Lipschitz equivalence $f$ from $X$ to $Y$. Then, Corollary 6.7 in \cite{DaletLancien}, which is an extension of results in \cite{GKL2001}, insures the existence of a universal constant $K>0$ and a constant $M>0$ (depending on $f$) so that for any $\eps>0$, there exists a norm $|\ |$ on $Y$ satisfying 
		$$\forall y\in Y,\ \|y\|_Y \le |y| \le M\|y\|_Y\ \ \text{and}\ \ \forall \sigma\in [0,1],\ \overline{\rho}_{|\ |}(KM^2\sigma)\le \overline{\rho}_{X}(\sigma)+\eps.$$
		With this result in hands, it is clear that characterization $(iii)$ of $\textsf{A}_p$ in Theorem \ref{atheorem} is stable under coarse Lipschitz equivalences. 
	\end{proof}
	
	We also have.
	\begin{theorem} For any $p \in (1,\infty)$, the class ${\emph{\textsf{N}}}_p$ is stable under coarse-Lipschitz equivalences.
	\end{theorem}
	
	\begin{proof}  Similarly to the previous proof, this is a direct consequence of Corollary 6.7 in \cite{DaletLancien} and characterization $(iii)$ of $\textsf{N}_p$ in Theorem \ref{ntheorem}. 
	\end{proof}
	
	Obviously the above argument can also be applied to prove that $\textsf{A}_\infty$ is stable under coarse Lipschitz equivalences, which, as we explained, was already known. 
	
	\begin{problem} In \cite{Kalton2013}, N. Kalton proved that for $1<p<\infty$, the class $\textsf{T}_p$ is not stable under uniform homeomorphisms. It is not known however whether the class $\textsf{T}_\infty$ is stable under coarse Lipschitz isomorphisms (or even uniform homeomorphisms). In fact, a positive answer would imply that a Banach space coarse Lipschitz equivalent to $c_0$ is linearly isomorphic to $c_0$, which is an important open question. Indeed, it is known that the class of
		all $\mathcal L_\infty$ spaces is stable under coarse Lipschitz equivalences \cite{HeinrichManckiewicz1982} and that a
		$\mathcal L_\infty$ subspace of $c_0$ is isomorphic to $c_0$ \cite{JohnsonZippin1972}. 
	\end{problem}
	

	\section{Examples}
	
	We gather in this section a few known examples of $\textsf{T}_\infty$ or $\textsf{A}_\infty$ spaces and related problems.
	
	\subsection{Non-separable uniformly flatenable spaces} 
	
	The first obvious examples of non-separable $\textsf{T}_\infty$ (or equivalently, AUF-renormable)  spaces are given by $c_0(\Gamma)$ spaces, with $\Gamma$ uncountable.
	
	\begin{proposition}\label{c_0(gamma)} For any set $\Gamma$, the space $c_0(\Gamma)$ equipped with its natural norm is AUF.
	\end{proposition}
	
	\begin{proof} It follows immediately from the definition of the norm of $c_0(\Gamma)$ that 
		$$\forall t\in (0,1)\ \ \overline{\rho}_{c_0(\Gamma)}(t)=0.$$
	\end{proof}
	
	The next result was already known (see the remark after the proof). We present first a proof using that ${\textsf{T}}_\infty$ is a 3SP.
	
	\begin{theorem}\label{C(K)} Let $K$ be a compact scattered space such that its Cantor derived set of order $\omega$, $K^{(\omega)}$ is empty. Then $C(K)$ is $\textsf{\emph{T}}_\infty$.
	\end{theorem}
	
	\begin{proof} We shall prove it by induction on $n\in \mathbb N$ such that $K^{(n)}=\emptyset$. If $n=1$, then $K'=\emptyset$ and $K$ is finite. Therefore $C(K)$ is finite-dimensional and thus is $\textsf{T}_\infty$. Assume that the statement is true for $n\in \mathbb N$ and that $K^{(n+1)}=\emptyset$. The subspace of $C(K)$ defined by $Y=\{f\in C(K),\ f_{|_{K'}}=0\}$ is clearly isometric to $c_0(K\setminus K')$ and by Proposition \ref{c_0(gamma)} is $\textsf{T}_\infty$. Let now $Q$ be the restriction mapping from $C(K)$ to $C(K')$. It follows from the Tietze extension theorem that $Q$ is onto. Since $Y$ is the kernel of $Q$, we have that $C(K')$ is isomorphic to $C(K)/Y$. By induction hypothesis, $C(K')$ and thus $C(K)/Y$ are $\textsf{T}_\infty$. It now follows from Theorem \ref{AUF3SP} that $C(K)$ is $\textsf{T}_\infty$.
	\end{proof}
	
	\begin{remark} As we already mentioned, this is not a new result. Let us indicate a few other ways to prove it.
		\begin{enumerate}
			\item Let $K$ be a compact space such that $K^{(n)}=\emptyset$, $n\in \mathbb N$. The dual of $C(K)$ is isometric to $\ell_1(K)$. Define the following equivalent norm on $\ell_1(K)$:
			$$\forall \mu \in \ell_1(K),\ \ |\mu|=\sum_{x\in K} \alpha_x |\mu(x)|,$$
			where $\alpha_x=2^{-i}$ with $0\le i\le n-1$ such that $x\in K^{(i)}\setminus K^{(i+1)}$. This formula comes from \cite{Lancien1995}, where it is proved that this norm is $1$-AUC$^*$ and is the dual norm of an equivalent norm on $C(K)$. So its predual norm is AUF. 
			
			\item Let $X$ be a separable subspace of $C(K)$ and denote $Y$ the closed sub-$\ast$-algebra of $C(K)$ generated by $X$. Then $Y$ is isometric to a space $C(L)$, where $L$ is a compact metrizable space such that $L^{(\omega)}=\emptyset$. It follows from \cite{BessagaPelczynski1960} that $Y$ is either finite dimensional or isomorphic to $c_0(\mathbb N)$. So $X$ is $\textsf{T}_\infty$ and we can apply the separable determination of $\textsf{T}_\infty$ (Theorem \ref{separabledetermination}) to deduce  that $C(K)$ is $\textsf{T}_\infty$. 
			
			\item We conclude with the most sophisticated argument. It is known that if $K$ is a compact space such that $K^{(\omega)}=\emptyset$, then $C(K)$ is Lipschitz isomorphic to some $c_0(\Gamma)$ (see \cite{DevilleGodefroyZizler1990}). On the other hand, being AUF renormable is stable under Lipschitz isomorphisms (\cite{GKL2000} for the separable case and \cite{DaletLancien} for the general case, or use separable determination and saturation).
		\end{enumerate}
		
	\end{remark}
	
	It is also important to mention that Theorem \ref{C(K)}  provides (only in the non-separable setting)  examples of $\textsf{T}_\infty$ spaces that are not isomorphic to a quotient or a subspace of a $c_0(\Gamma)$ space. Indeed we have 
	
	\begin{theorem} There exists a compact space $K$ such that $K^{(3)}=\emptyset$, but $C(K)$ is not isomorphic to a quotient of a subspace of a $c_0(\Gamma)$ space.
	\end{theorem}

	Let us indicate this now classical construction. There exists a scattered separable uncountable compact space $K$ so that $K^{(3)}=\emptyset$. This space is often called the Mr\'owka-Isbell space and its construction is due to Mr\'owka \cite{Mrowka} and Isbell (credited in \cite{GillmanJerison}). We also refer to its description in \cite{ZizlerHandbook} page 1757, where its construction is  based on the Johnson-Lindenstrauss space $JL_0$ \cite{JohnsonLindenstrauss1974}. Since $K$ is separable, $C(K)$ admits a countable family of separating functionals (the Dirac maps at the points of a dense countable subset of $K$). But $C(K)$ is not separable, as $K$ is uncountable and scattered and therefore non metrizable. It follows that $C(K)$ is not weakly Lindelöf determined (WLD in short): see Theorem 5.37 and Proposition 5.40 in \cite{HMVZ}, or see \cite{VWZ1994}. We conclude by recalling that $c_0(\Gamma)$ is always WLD and that being WLD is stable by passing to subspaces or quotients (see also \cite{HMVZ} and references therein).
	
	\begin{problem} 
		We do not know if there exists a $\textsf{T}_\infty$ space which is not isomorphic to quotient of a subspace of a $C(K)$ space with $K^{(\omega)}=\emptyset$. 
	\end{problem}
	
	\subsection{An interesting $\textsf{A}_\infty$ space}
	
	We already explained that showing that $\textsf{T}_\infty$ is stable under coarse Lipschitz equivalences would imply that a Banach space coarse Lipschitz equivalent to $c_0$ is linearly isomorphic to $c_0$. At this point it is only known that a Banach space coarse Lipschitz equivalent to $c_0$ is $\textsf{A}_\infty$ and $\mathcal L_\infty$. Another hope was to show that a separable Banach space which is $\textsf{A}_\infty$ and $\mathcal L_\infty$ is necessarily $\textsf{T}_\infty$ (see conjecture after Theorem 5.6 in \cite{GKL2001}). Let us mention here that this question has been solved negatively by Argyros, Gasparis, and Motakis in \cite{AGM2016}, who showed the existence of a separable Banach space $X$ which is $\textsf{A}_\infty$ and $\mathcal L_\infty$ but so that every infinite dimensional subspace of $X$ contains an infinite dimensional reflexive subspace. 
	
	\medskip
	{\bf Acknowledgments.} We thank Petr H\'{a}jek for valuable discussions on non-separable $C(K)$ spaces and Gilles Godefroy for pointing to us the important example by Argyros, Gasparis and Motakis. We also thank the referee for her/his suggestions that helped improve the presentation of this article.

	\bibliographystyle{amsplain}

\end{document}